\documentclass{birkmult}
\usepackage{graphicx,amssymb}
\usepackage{amsmath,amssymb,amscd}
\input xy
\xyoption{all}
\usepackage[all]{xy}
\usepackage{hyperref}

\newcommand{\ra}{\rightarrow}

\newcommand{\cO}{\mathcal{O}}

\newcommand{\rk}{\mbox{rk}}
\theoremstyle{plain}
\newtheorem{theorem}{Theorem}[section]
\newtheorem{lem}[theorem]{Lemma}
\newtheorem{prop}[theorem]{Proposition}
\newtheorem{cor}[theorem]{Corollary}
\newtheorem{conj}[theorem]{Conjecture}
\newtheorem{pr}[theorem]{Problem}
\newtheorem{rem}[theorem]{Remark}

\numberwithin{equation}{section}
\begin{document}
\title[Clifford Indices]{Clifford Indices for Vector Bundles on Curves}

\author{H. Lange}

\address{Mathematisches Institut\\
              Universit\"at Erlangen-N\"urnberg\\
              Bismarckstra\ss e $1\frac{ 1}{2}$\\
              D-$91054$ Erlangen\\
              Germany}
              \email{lange@mi.uni-erlangen.de}
\author{P. E. Newstead}
\address{Department of Mathematical Sciences\\
              University of Liverpool\\
              Peach Street, Liverpool L69 7ZL, UK}
\email{newstead@liv.ac.uk}

\thanks{Both authors are members of the research group VBAC (Vector Bundles on Algebraic Curves). The second author 
would like to thank the Mathematisches Institut der Universit\"at 
         Erlangen-N\"urnberg for its hospitality}

\begin{abstract}
For smooth projective curves of genus $g\ge4$, the Clifford index is an important invariant which provides a bound for the dimension of the
space of sections of a line bundle.  This is the first step in distinguishing curves of the same genus. In this 
paper we generalise this to introduce Clifford indices for semistable vector bundles on curves. We study these invariants,
giving some basic properties and carrying out some computations for small ranks and for general and some special curves. 
For curves whose classical Clifford index is two, we compute all values of our new Clifford indices.
\end{abstract}
\maketitle

\section{Introduction}

Let $C$ be a smooth projective curve of genus $g\ge4$ defined over an algebraically closed field of characteristic $0$. If $L$ is a line bundle on $C$ with space of sections $H^0(L)$ of dimension $h^0(L)\ge2$, then evaluation of sections defines a morphism $C\to {\mathbb P}^{(h^0(L)-1)}$. These morphisms yield much information about the geometry of $C$, in particular about the possible projective embeddings of $C$ and the syzygies resulting from such embeddings. It is therefore important to obtain precise upper bounds on $h^0(L)$ in terms of the degree of $L$. These upper bounds depend on the curve, not just on the value of $g$, and a first measure of the possible bounds is given by the Clifford index of $C$, whose definition runs as follows. 
We consider line bundles $L$ on $C$ and define the {\em Clifford index} $\gamma_1$ of $C$ by
$$
\gamma_1=\min_L\left\{\deg L-2(h^0(L)-1)\;|\;h^0(L)\ge 2,h^1(L)\ge2\right\}
$$
or equivalently
$$
\gamma_1=\min_L\left\{\deg L-2(h^0(L)-1)\;|\;h^0(L)\ge 2,\deg L\le g-1\right\}.
$$
(The equivalence of the two definitions follows from Serre duality. The reason for requiring $g\ge4$ is to ensure the existence of line bundles as in the definition.)

It is natural to generalise this to higher rank and in particular to semistable (or stable) vector bundles. The restriction to semistable bundles is natural as it allows for restrictions on the dimension of the space of sections essentially identical to those which exist for line bundles. In particular Clifford's Theorem has been extended to semistable bundles; the simplest proof of this is due to 
G. Xiao and appears as \cite[Theorem 2.1]{bgn}. Semistable bundles also arise naturally in the study of syzygies in connection with conjectures of Green and Lazarsfeld \cite{g,gl} (in this context, see \cite{t4}). Moreover the moduli spaces of semistable bundles are objects of interest in their own right. We may note that the existence of semistable bundles with specified numbers of sections has recently been used in two papers \cite{gt,cgt} which obtain new information on the base locus of the generalised theta-divisor; the first of these papers extends results of Arcara \cite{a}, Popa \cite{p} and Schneider \cite{s}, while the second uses also the strange duality theorem, recently proved in \cite{be1,be2,mo}.

A key r\^ole in some of these papers is played by the evaluation sequence
\begin{equation}\label{eq0.1}
0\ra  E_L^*\ra H^0(L)\otimes{\mathcal O}\ra L\ra 0, 
\end{equation}
where $L$ is a line bundle generated by its sections. For our purposes the key issue is the stability of $E_L$, which was proved for $L=K_C$ in \cite{pr} and for $\deg L\ge2g+1$ by Ein and Lazarsfeld \cite{el}. Subsequently David Butler \cite{b,bu} considered more generally the sequence
\begin{equation}\label{eq0.2}
0\ra M_{V,E}^*\ra V\otimes{\mathcal O}\ra E\ra0,
\end{equation}
where $E$ is a vector bundle and $V$ is a linear subspace of $H^0(E)$ which generates $E$. (The construction of \eqref{eq0.2} has come to be known as the \emph{dual span construction}.) Butler showed \cite[Theorem 1.2]{b} that, if $V=H^0(E)$ and $E$ is semistable (stable) of slope $\ge2g$, then $M_{V,E}$ is semistable (stable) with a minor exception in the case of stability when the slope is equal to $2g$. Results for line bundles of smaller degree have been obtained in \cite{bu,bp,bbn,bo}. There is an important conjecture of Butler (which we discuss briefly in the final section) that the bundles $M_{V,E}$ are semistable ``in  general''.

The Brill-Noether locus $B(n,d,k)$ in the moduli space of stable bundles of rank $n$ and degree $d$ on $C$ is comprised of those bundles $E$ for which $h^0(E)\ge k$. These loci have been studied for around 20 years and a good deal 
is known about their non-emptiness. However almost all the results are either true 
for any $C$ \cite{bgn,bmno,m1,m2} or for general $C$ only \cite{t1,t2,t3}. Precise results are known for hyperelliptic curves 
\cite[Section 6]{bmno} and for bielliptic curves \cite{ba} but little has been done on other special curves. 
The main exceptions to this are papers of R. Re \cite{re} and V. Mercat \cite{m} and a recent paper 
by L. Brambila-Paz and A. Ortega \cite{bo}.  These papers use only the classical Clifford index $\gamma_1$ which is defined using line bundles, although in many respects \cite{m} is the staring point for our investigations.

In \cite{ba2}, E. Ballico gave five definitions of Clifford indices for vector bundles but did not develop the concept to any significant extent. We give two definitions, both using semistable bundles (whereas Ballico used indecomposable and stable bundles). Our definitions differ from those of Ballico in other respects as well, in that we do not assume that our bundles are generated by their sections (although in fact most of our examples are so generated) and we consider only bundles whose slope is at most $g-1$ (whereas Ballico requires only that $h^1(E)\ne0$). In fact we define, for any vector bundle $E$ of rank $r$ and degree $d$,
$$
\gamma(E) := \frac{1}{n} \left(d - 2(h^0(E) -n)\right) = \mu(E) -2\frac{h^0(E)}{n} + 2,
$$
where $\mu(E)=\frac{d}n$.
(Ballico defines $\text{Cliff}(E)$ in the same way but without the scaling factor $\frac{1}{n}$.)
We then define $\gamma_n$ to be the minimum value of $\gamma(E)$ for semistable bundles $E$ of rank $n$ with $h^0(E)\ge n+1$ and $\mu(E)\le g-1$. Our second index $\gamma_n'$ is defined similarly, but with $h^0(E)\ge2n$. For line bundles, the two definitions coincide  
and reduce to the classical Clifford index. The use of semistable, rather than stable, bundles is likely to give better specialisation properties, although the question of Clifford 
indices for stable bundles is also of interest and will undoubtedly be investigated further in the future. It may be noted that there are results in the literature giving bounds on $h^0(E)$ 
for indecomposable bundles \cite{tan} and also for bundles 
of rank $\le3$ \cite{cs,ln} in terms of degrees of stability, but, for the reasons stated earlier, we feel that semistable bundles form the most natural context for these ideas.

Another natural question to ask is why we use at least $n+1$ independent sections for the definition of $\gamma_n$. The first reason is that the question of non-emptiness of Brill-Noether loci has been completely solved for $h^0(E)\le n$ (see \cite{bgn,bmno}) and depends only on the genus of $C$. More fundamentally, the existence of semistable bundles with $h^0(E)\ge n+1$ is closely linked with the non-emptiness of certain Quot schemes and the existence of stable maps from $C$ to a Grassmannian (see \cite{pro}), both of which have implications for the geometry of $C$. 

We now summarise the contents of the paper.
In Section 2, we give the definitions of $\gamma_n$ and $\gamma_{n}'$ and obtain some 
elementary properties. In Section 3, we relate our invariants to the conjecture of Mercat, 
which is a strengthened version of the assertion that $\gamma_{n}'=\gamma_1$.  We then make 
some deductions from the results of \cite{m}, including an almost complete determination 
of the values of $\gamma_n$ for $n\ge g-3$ (Theorem \ref{thm3.5}). 

Section 4 is the central section of the paper. In it, we introduce the invariants 
$$
d_r := \min \{ \deg L \;|\; L \; \mbox{a line bundle on} \; C \; \mbox{with} \; h^0(L) \geq r +1\},
$$ 
which form the {\em gonality sequence} of $C$; these invariants play an important r\^ole 
in the theory of special curves and are completely known in many cases. We describe 
the properties of these invariants which we need later in the paper. We then introduce 
the dual span construction \eqref{eq0.1}. 
We verify Butler's conjecture in the case of line bundles of degree $d_n$ under certain conditions 
on the gonality sequence (Proposition \ref{prop5.1}). This allows us to prove our first main theorem.

\medskip
\noindent{\bf Theorem \ref{pr4.9}.}
\emph{Let $E$ be a semistable bundle of rank $n$ and degree $d_n$.} 

(a): \emph{If $\frac{d_p}{p} \geq \frac{d_{p+1}}{p+1}$ for all $p < n$ and $d_n \neq nd_1$, then
$$
h^0(E) \leq n+1
$$
and there exist semistable bundles of rank $n$ and degree $d_n$ with $h^0 = n+1$.}

(b): \emph{If $d_n = nd_1$, then 
$$
h^0(E) \leq 2n
$$
and there exist semistable bundles of rank $n$ and degree $d_n$ with $h^0 = 2n$.}

(c): \emph{If $\frac{d_p}{p} \geq \frac{d_n}{n}$ for all $p < n$ and $E$ is stable, then 
$$
h^0(E) \leq n+1.
$$}

\noindent
As a corollary (Corollary \ref{cor4.16}) we show that Mercat's conjecture holds for semi\-stable 
bundles of rank $n$ and degree $\le d_n$, again under certain conditions on the gonality sequence.  
We also complete the computation of $\gamma_n$ for $n\ge g-3$ (Theorem \ref{thm4.14}). 
For a curve with $\gamma_1=2$, this allows us to compute all values of $\gamma_n$ 
(Corollary \ref{cor4.15}). We complete the section by obtaining an upper bound for $\gamma_n$ and 
lower bounds for $\gamma(E)$ dependent on the existence of certain subbundles. 

In Section 5, we prove the following two theorems for bundles of rank $2$.

\medskip
\noindent{\bf Theorem \ref{prop4.7}.}
\emph{$\quad \gamma_2 = \min \left\{ \gamma_2', \frac{d_2}{2} - 1\right\}$.}

\medskip
\noindent{\bf Theorem \ref{prop4.8}.}
\emph{$\quad \gamma_2' \geq \min \left\{ \gamma_1, \frac{d_4}{2} -2 \right\}$.}

\noindent These theorems yield the precise formula $\gamma_2=\min\left\{\gamma_1,\frac{d_2}2-1\right\}$ (Corollary \ref{cor4.9}). In particular, $\gamma_2$ is not determined by $\gamma_1$.

\medskip
In Sections 6 and 7, we extend this partially to ranks $3$, $4$ and $5$, obtaining the following results.

\medskip
\noindent{\bf Theorem \ref{prop5.4}.}
\emph{Suppose $\frac{d_2}{2} \geq \frac{d_3}{3}$. Then
$$
\gamma_3 = \min \left\{ \gamma_3', \frac{1}{3}(d_3 - 2) \right\}.
$$
}

\medskip
\noindent{\bf Theorem \ref{prop7.1}.}
\emph{If $ \frac{d_3}{3} \geq \frac{d_4}{4}$, then
$$
\gamma_4 = \min \left\{ \gamma_4', \frac{1}{4}(d_4 -2), \frac{1}{2}(d_2-2) \right\}.
$$
} 

\medskip
\noindent{\bf Theorem \ref{prop8.1}.}
\emph{If $ \frac{d_p}{p} \geq \frac{d_{p+1}}{p+1}$ for $1 \leq p \leq 4$, then
\begin{eqnarray*}
\gamma_5 &\geq  \min &\left\{  \gamma_5', \frac{1}{2}(d_2 -2), \frac{1}{5}(d_5 -2),  \frac{1}{5}(d_1 + 2d_2 - 6), 
\frac{1}{5}(d_1 + d_4 -4), \right.\\
&& \left. \hspace*{1.2cm} \frac{1}{5}(2d_1 + d_3 - 6), \frac{1}{5}(3d_1 + d_2 -8) , \frac{1}{5}(d_2 + d_3 -5) \right\}. 
\end{eqnarray*}
}

\noindent This last result looks weaker than we would hope, but we show that for a general 
curve it gives the following much more precise result.

\medskip
\noindent{\bf Corollary \ref{cor7.6}}
\emph{Let $C$ be a general curve. Then
$$
\gamma_5 = \min \left\{ \gamma_5', \frac{1}{5}\left(g - \left[ \frac{g}{6} \right] + 3\right) \right\}.
$$
}

In Section 8 we consider smooth plane curves. In this case we know the gonality sequence 
precisely by a theorem of Noether. Although such curves do not satisfy all the conditions mentioned 
earlier, we are able to carry out the same analysis and to obtain good results for $n\le5$.

In the final section, we discuss some problems.

Our main arguments depend on a result of Paranjape and Ramanan \cite[Lemma 3.9]{pr} 
and on Mercat's paper \cite{m} as well as on the dual span construction. We have also 
made much use of results on special curves due to Gerriet Martens and his collaborators. 
We are grateful to him for some useful discussions and for drawing our attention to a number of papers. 

Throughout the paper $C$ will be a smooth curve of genus $g\ge4$ defined over an algebraically closed field 
of characteristic 0. 
We recall that, for a vector bundle $E$ of rank $n$ and degree $d$, the {\em slope} $\mu(E)$ is 
defined by $\mu(E):=\frac{d}n$.

We are grateful to the referee for pointing out the reference \cite{ba2}.

\section{Definition of $\gamma_n$ and $\gamma_n'$}

Let $C$ be a smooth projective curve of genus $g \geq 4$. For any vector bundle $E$ of rank $n$ and degree $d$ on $C$ consider
$$
\gamma(E) := \frac{1}{n} \left(d - 2(h^0(E) -n)\right) = \mu(E) -2\frac{h^0(E)}{n} + 2.
$$
The proof of the following lemma is a simple computation.
\begin{lem} \label{lem1}
$\gamma(K_C \otimes E^*) = \gamma(E)$. \hspace*{6.2cm}$\square$
\end{lem}
For any positive integer $n$ we 
define the following invariants of $C$:
$$
\gamma_n := \min_{E} \left\{ \gamma(E) \left|  
\begin{array}{c}   E \;\mbox{semistable of rank}\; n \\
h^0(E) \geq n+1,\; \mu(E) \leq g-1
\end{array} \right\} \right.
$$
and
$$
\gamma_n' := \min_{E} \left\{ \gamma(E) \;\left| 
\begin{array}{c} E \;\mbox{semistable of rank}\; n \\
h^0(E) \geq 2n,\; \mu(E) \leq g-1
\end{array} \right\}. \right.
$$
Note that $\gamma_1 = \gamma_1'$ is the usual Clifford index of the curve $C$. We say that $E$ {\it contributes to} 
$\gamma_n$ (respectively $\gamma_n'$) if $E$ is semistable of rank $n$ with $\mu(E) \leq g-1$ and $h^0(E) \geq n+1$ 
(respectively $h^0(E) \geq 2n$). If in addition 
$\gamma(E) = \gamma_n$ (respectively $\gamma(E) = \gamma_n'$), we say that $E$ {\it computes} $\gamma_n$ 
(respectively $\gamma_n'$).

\begin{lem} \label{lemma2.2}
If $p|n$, then $\gamma_n \leq \gamma_p$ and 
$\gamma_n \leq \gamma_n' \leq \gamma_p'$.
\end{lem}

\begin{proof}
Let $E$ be a vector bundle computing $\gamma_p$. Then $\gamma(\oplus_{i=1}^{\frac{n}{p}}E) = \gamma_p$
which gives the first assertion.
It is obvious that $\gamma_n \leq \gamma_n'$. The proof of the last inequality 
is the same as the proof of the first statement.
\end{proof}

\begin{lem}  \label{lem2.1}
$$
0 \leq \gamma_n \leq \frac{1}{n} \left(g - \left[\frac{g}{n+1} \right] +n-2 \right). 
$$
\end{lem}

\begin{proof}
If $E$ is a vector bundle computing $\gamma_n$, we have by 
\cite[Theorem 2.1]{bgn} that $h^0(E) \leq \frac{\deg E}{2} + n$, which implies $\gamma(E) \geq 0$. Hence $\gamma_n \geq 0$.

From \cite[Theorem 2]{bu} or \cite[Proposition 4.1 (ii)]{bp} we know that on a general curve there exist semistable vector bundles of rank $n$ and degree 
$d = g - \left[\frac{g}{n+1} \right] + n$ with $h^0(E) \geq n+1$. Since $g \geq 4$, we have $\mu(E) \leq g-1$. 
By semicontinuity this is valid on any curve $C$.
Hence
$$
\gamma_n \leq \frac{1}{n} \left( g - \left[ \frac{g}{n+1} \right] + n-2 \right).
$$ 
\end{proof}

\begin{cor} \label{cor2.3}
Suppose $g \geq 7$. For a general curve $C$ and every $n \geq 3$, we have
$$
\gamma_n < \gamma_1.
$$
\end{cor}

\begin{proof}
On a general curve we know $\gamma_1 = \left[ \frac{g-1}{2} \right]$. According to Lemma \ref{lem2.1} it suffices to show
\begin{equation} \label{eqn2.1}
\frac{1}{n} \left(g - \left[\frac{g}{n+1} \right] + n - 2 \right) < \left[\frac{g-1}{2} \right].
\end{equation}
Suppose $n \geq 3$.
Since $\frac{g}{n+1} - 1 < [\frac{g}{n+1}]$ and $\frac{g}{2} - 1 \leq [\frac{g-1}{2}]$, this is implied by
$$
\frac{1}{n} \left(g -\frac{g}{n+1} +1 +n - 2 \right) \leq \frac{g}{2} - 1, 
$$ 
which is equivalent to
$$
g \geq \frac{(4n-2)(n+1)}{(n+1)(n-2) + 2} = 4 + \frac{6n -2}{n^2 - n}.
$$
This is valid for $g \geq 7$. 
\end{proof}

\begin{rem}
\emph{Corollary \ref{cor2.3} remains valid for $g=5$ and for $g = 6,\; n \geq 4$ (for $g=5, \; n\leq 6$ one needs to check 
directly in \eqref{eqn2.1}). The corollary is also valid for $n = 2$, provided $g \geq 7, \; g \neq 8$.
In fact, for $g \geq 9$ the same proof works. The case $g = 7$ can be checked from \eqref{eqn2.1}.}
\end{rem}

\begin{prop} \label{prop3.2}
\emph{(a)} If $\gamma_1 = 0$ or $1$, then for all $n$,
$$
\gamma_n = \gamma_n' = \gamma_1.
$$
\emph{(b)} If $\gamma_1 \geq 2$, then $\gamma_n \geq 1$ for all $n$.
\end{prop}

\begin{proof}
(a): By Lemma \ref{lem2.1} we have $\gamma_n \geq 0$. So the result for $\gamma_1 = 0$ follows by Lemma \ref{lemma2.2}. 
Suppose $\gamma_1 = 1$.
If $\gamma_n < 1$, then there exists a semistable bundle $E$ with $h^0(E) \geq n+1$ and degree $d \leq n(g-1)$ such that 
$$
d - 2(h^0(E) -n) < n.
$$
So 
$$
h^0(E) > \frac{d+n}{2}.
$$
If $d \geq n$, this contradicts \cite{re}. If $d < n$, then $h^0(E) < n$ by \cite{bgn}. So $\gamma_n \geq 1$ and hence 
$\gamma_n = \gamma_n' = \gamma_1$.

(b): The argument in the proof of (a) for $\gamma_1 = 1$ is valid also for $\gamma_1 \geq 2$.
\end{proof}

\begin{cor} \label{cor2.7}
If $\gamma_1 \geq 1$, then $\lim_{n \ra \infty} \gamma_n = 1$.
\end{cor}

\begin{proof}
This follows from Proposition \ref{prop3.2} and Lemma \ref{lem2.1}.
\end{proof}

\section{Mercat's conjecture}

We want to relate the invariants $\gamma_n'$ with Mercat's conjecture (see \cite{m}), which can be stated as follows:

\begin{conj} \label{conj3.1}
Let $E$ be a semistable vector bundle of rank $n$ and degree $d$.\\
\emph{(i)} If $\gamma_1 + 2 \leq \mu(E) \leq 2g - 4 - \gamma_1$, then $h^0(E) \leq \frac{d-\gamma_1 n}{2} + n$.\\
\emph{(ii)} If $1 \leq \mu(E) \leq \gamma_1 + 2$, then $h^0(E) \leq \frac{1}{\gamma_1 +1}(d-n) + n$.\\
\end{conj}

\begin{lem} \label{lem2}
Conjecture \ref{conj3.1} (i) is equivalent to 

\emph{(i$'$)} If $\gamma_1 + 2 \leq \mu(E) \leq 2g - 4 - \gamma_1$, then $\gamma(E) \geq \gamma_1$.
\end{lem}
\begin{proof} Suppose (i) holds. Then
$$
\gamma(E) = \mu(E) -2\frac{h^0(E)}{n} + 2 \geq \mu(E) - 2\frac{\frac{d-\gamma_1 n}{2} +n}{n} + 2 = \gamma_1.
$$
The converse implication follows by the same computation. 
\end{proof}

\begin{prop} \label{prop3.3}
Let $n \geq 2$ be an integer.\\
\emph{(a):} Conjecture \ref{conj3.1} implies the equality $\gamma_n' = \gamma_1$.\\
\emph{(b):} The equality $\gamma_n' = \gamma_1$ implies conjecture \ref{conj3.1} (i).
\end{prop}

\begin{proof}
(a) Assume Conjecture \ref{conj3.1} holds and suppose $E$ contributes to $\gamma_n'$. According to Lemma \ref{lemma2.2}
we have to show that $\gamma(E) \geq \gamma_1$.

If $\mu(E) \geq \gamma_1 + 2$, Lemma \ref{lem2} implies the assertion. So suppose $1 \leq \mu(E) < \gamma_1 +2$.
Then by (ii),
$$
h^0(E) \leq \frac{1}{\gamma_1 + 1}(d-n) + n < \frac{1}{\gamma_1 + 1}\big( n(\gamma_1 + 2) -n\big) + n = 2n,
$$
a contradiction.

(b) Assume that $\gamma_n' = \gamma_1$ and consider a semistable vector bundle 
$E$ of rank $n$ with $\gamma_1 + 2 \leq  \mu(E) \leq 2g-4 -\gamma_1$. By Lemma \ref{lem2} we have to show that
$\gamma(E) \geq \gamma_1$. In view of Lemma \ref{lem1} we can assume that $\mu(E) \leq g-1$.

If $h^0(E) \geq 2n$, then $\gamma(E) \geq \gamma_n' = \gamma_1$ by assumption. 
If $h^0(E) < 2n$, then 
$$
\gamma(E) = \mu(E) - 2\frac{h^0(E)}{n} + 2 > \gamma_1 + 2 - 2\frac{2n}{n} + 2 = \gamma_1.
$$
\end{proof}

\begin{rem}
\emph{ If Conjecture \ref{conj3.1} (ii) holds and $1 \leq \mu(E) < \gamma_1 + 2$, then $h^0(E) < 2n$. So $E$ does not contribute to $\gamma_n'$}.
\end{rem}

\begin{prop} \label{prop3.7}
If $\gamma_1 \geq 2$, then $\gamma_n' \geq 2$ for all $n$. 
\end{prop}

Note that $\gamma_1 \geq 2$ implies $g \geq 5$.

\begin{proof}
We use \cite[Theorem 1]{m}. Let $E$ be a semistable bundle of rank $n$ and degree $d$. If $d < (2 + \frac{2}{g-4})n$, then by 
\cite[Theorem 1 (ii)]{m} we have 
\begin{eqnarray*}
h^0(E) & \leq & \frac{1}{g-2}(d-n) + n\\
& < & \frac{1}{g-2} \left( n + \frac{2n}{g-4} \right) + n\\
& = & n \left(1 + \frac{1}{g-4} \right) \leq 2n. 
\end{eqnarray*}
So $E$ does not contribute to $\gamma_n'$. Now \cite[Theorem 1 (i)]{m} implies $\gamma_n' \geq 2$.
\end{proof}

We can use Mercat's results of \cite{m1}, \cite{m2} and \cite{m} to obtain the following theorem.

\begin{theorem} \label{thm3.5}
Let $C$ be a curve with Clifford index $\gamma_1 \geq 2$.\\
\emph{(a):} If $n > g$, then
$$
\gamma_n = 1 + \frac{g-2}{n};
$$
\emph{(b):} If $n = g$, then 
$$
\gamma_n \left\{ \begin{array}{cccc}
                 = & 2 - \frac{2}{g} & for & g \geq 6,\\
                 \geq & \frac{7}{5} & for & g = 5;
                 \end{array} \right.
$$                 
\emph{(c):} If $n = g-1$, then 
$$
\gamma_n = 2 - \frac{2}{g-1};
$$
\emph{(d):} If $g-3 \leq n \leq g-2$, then 
$$
\gamma_n \geq 2 - \frac{1}{n};
$$
\emph{(e):} If $n \leq g-4$, then 
$$
\gamma_n \geq 2.
$$
\end{theorem}

\begin{proof} 
Let $E$ be a semistable bundle of rank $n$ and degree $d$. So $ \mu = \mu(E) = \frac{d}{n}$. We consider 4 cases:

{\it Case 1}: $1 < \mu < 2$: By \cite{m1}, $h^0(E) \leq n + \frac{1}{g}(d-n)$ and so
\begin{equation*} 
 \gamma(E) \geq \frac{1}{n}\left(d - \frac{2}{g}(d-n)\right).
\end{equation*}

{\it Case 2}: $\mu = 2$: By \cite{m2}, $h^0(E) \leq n+[\frac{n}{g-1}]$ and so 
\begin{equation*} 
\gamma(E) \geq \frac{1}{n} \left(2n - 2 \left[\frac{n}{g-1} \right]  \right).
\end{equation*}

{\it Case 3}: $2 < \mu <2+ \frac{2}{g-4}$: By \cite[Theorem 1 (ii)]{m}, $h^0(E) \leq n + \left[\frac{1}{g-2}(d-n) \right]$ and so
\begin{equation} \label{eq3.1}
\gamma(E) \geq \frac{1}{n}\left( d - 2 \left[ \frac{d-n}{g-2}\right] \right).
\end{equation}

{\it Case 4}: $2 + \frac{2}{g-4} \leq \mu$: By \cite[Theorem 1 (i)]{m}, 
 \begin{equation} \label{eq3.2}
\gamma(E) \geq 2.
 \end{equation}

In Case 1 the right hand side is an increasing function of $d$. 
So we need to look for the smallest $d$ in the given range for which a bundle $E$ exists with $h^0(E) \geq n+1$. 
We must have $d = n+g$ and this is in the required range if $n>g$ and then for such $E$,
\begin{equation} \label{eq3.3}
\gamma(E) = 1 + \frac{g-2}{n} < 2.
\end{equation}
When $n > g$, such $E$ always exists (see \cite{m1}). 

By \cite{m2}, Case 2 always occurs provided $n \geq g-1$ and gives us bundles $E$ with
\begin{equation}  \label{eq3.4} 
\gamma(E) = 2 - \frac{2}{n} \left[ \frac{n}{g-1} \right] < 2.
\end{equation}
It remains to deal with Case 3. The smallest value of the right hand side of \eqref{eq3.1} is 
given by one of the following three possibilities:
\begin{itemize}
\item $d = 2n+1$ and $n \geq g-3$,
\item $d = 2n+2$ and $ g-2 \; \mbox{divides} \; n+2$,
\item $d = n + g - 2$.
\end{itemize}
If none of these possibilities occurs within the range $2 < \mu < 2 + \frac{2}{g-4}$, then Case 3 does not arise. 

For $d = 2n+1$ we get
\begin{equation} \label{eq3.5}
\gamma(E) \geq 2 - \frac{1}{n} \left( 2 \left[ \frac{n+1}{g-2} \right] - 1 \right)
\end{equation}
and we require $2n+1 < n(2 + \frac{2}{g-4} )$, i.e. $n > \frac{g-4}{2}$ which is true since $n \geq g-3$.

For $d = 2n+2$ we get
\begin{equation} \label{eq3.6}
\gamma(E) \geq 2 - \frac{2}{n} \left( \frac{n+2}{g-2} - 1 \right). 
\end{equation}
and we require that $g-2 \; \mbox{divides} \;n+2$ and $2n+2 < n(2 + \frac{2}{g-4} )$, i.e. $n > g-4$.

For $d = n + g-2$ we get
\begin{equation} \label{eq3.7}
\gamma(E) \geq 1 + \frac{g-4}{n}
\end{equation}
and we require $2n < n+g-2 < 2n + \frac{2n}{g-4}$. i.e. $n = g-3$. In this case \eqref{eq3.7} gives the same inequality as 
\eqref{eq3.5} and hence can be ignored.

If $n > g$, the right hand side of \eqref{eq3.3} is less than or equal to the right hand sides of \eqref{eq3.4}, 
\eqref{eq3.5} and \eqref{eq3.6}.
So for $n > g$ we obtain
$$
\gamma_n \geq \min \left\{2, 1 + \frac{g-2}{n} \right\} = 1 + \frac{g-2}{n}
$$
and this can be attained by a bundle $E$ of degree $n+g$ with $h^0(E) = n+1$.

For $n= g \geq 6$ we get from \eqref{eq3.2}, \eqref{eq3.4}, \eqref{eq3.5} and \eqref{eq3.6}, 
$$
\gamma_n \geq \min \left\{2, 2 - \frac{2}{g}, 2 - \frac{1}{g} \right\} = 2 - \frac{2}{g}
$$ 
and this can be attained by a bundle $E$ of degree $2g$ with $h^0(E) = g+1$. For $n=g=5$ we get
$$
\gamma_5 \geq \min \left\{ 2,2-\frac{2}{5}, 2 - \frac{3}{5} \right\} = \frac{7}{5}.
$$
For $n=g-1$ we get from \eqref{eq3.2}, \eqref{eq3.4}, \eqref{eq3.5} and \eqref{eq3.6},
$$
\gamma_{g-1} \geq \min \left\{2, 2 - \frac{2}{g-1}, 2 - \frac{1}{g-1} \right\} = 2 - \frac{2}{g-1}
$$
and the bound is attained by a bundle of degree $2g-2$ with $h^0(E) = g$. (In fact, the unique 
such semistable bundle is the dual span of the canonical bundle $K_C$ \cite[Theorem 1]{m2}). 

For $n = g-2$ or $g-3$ we get from \eqref{eq3.2} and \eqref{eq3.5},
$$
\gamma_n \geq 2 - \frac{1}{n} \left( 2 \left[ \frac{n+1}{g-2} \right] -1 \right) = 2 - \frac{1}{n}.
$$
For $n \leq g-4$ none of the inequalities \eqref{eq3.3} to \eqref{eq3.6} applies. So there is no semistable $E$ of rank $n \leq g-4$ 
with $\mu(E) < 2 + \frac{2}{g-4}$ and $h^0(E) \geq n+1$. Hence
$$
\gamma(E) \geq 2
$$
by \eqref{eq3.2}.
\end{proof}

\begin{rem}
\emph{Note that Theorem \ref{thm3.5} (a) gives a more precise version of Corollary \ref{cor2.7}.}
\end{rem}

\begin{prop} \label{prop3.6}
If $\gamma_1 \geq 3$, then 
$$
\gamma_2 \geq \min \left\{ \gamma_1, \frac{\gamma_1}{2} + 1 \right\} \quad \mbox{and} \quad \gamma_2' \geq \min \left\{ \gamma_1, \frac{\gamma_1}{2} + 2 \right\}.
$$
In particular, $\gamma_2' = \gamma_1$ for $\gamma_1 \leq 4$.
\end{prop}

\begin{proof}
Suppose $E$ is semistable of rank 2 and degree $d$. If $3\gamma_1 -1 \leq d \leq 2g-2$, then by \cite[Corollary 3]{m},
$$
\gamma(E) \geq \gamma_1.
$$
If $d \leq 3 \gamma_1 -2$, then by \cite[Lemma 5]{m}, $h^0(E) \leq \frac{d-\gamma_1}{4} +2$. So
$$
\gamma(E) \geq \frac{d + \gamma_1}{4}.
$$
In the last case $E$ can contribute to $\gamma_2$ only if $d \geq \gamma_1 + 4$ and to $\gamma_2'$ only if $d \geq \gamma_1 + 8$. 
This gives the assertion.
\end{proof}

\section{The invariants $d_r$}

For any positive integer $r$ we define the invariant $d_r$ of the curve $C$ by
$$
d_r := \min \{ \deg L \;|\; L \; \mbox{a line bundle on} \; C \; \mbox{with} \; h^0(L) \geq r +1\}.
$$
Note that $d_1$ is the gonality of $C$, $d_2$ is the minimal degree of a non-degenerate rational map of the curve $C$ into the projective
plane etc. 
We refer to the sequence $d_1,d_2,\ldots$ as the {\it gonality sequence} of $C$.
We say that $L$ {\it computes} $d_r$ if $\deg L = d_r$ and $h^0(L) \geq r+1$. We say also that $d_r$ 
{\it computes} $\gamma_1$ if $\gamma_1 = d_r -2r$.

\begin{rem}  \label{rem4.1}
\emph{The gonality sequence is usually defined only for those $r$ for which $d_r \leq g-1$ (see \cite[Digression (3.5)]{ckm}),
 but for 
our purposes the definition above is more convenient. If $k = d_1$ computes $\gamma_1$ (i.e. $d_1 = \gamma_1 + 2$), 
the curve $C$ is called $k$-{\it gonal}.}
\end{rem}

\begin{lem} \label{lem4.1}
\emph{(a)} $d_r < d_{r+1}$ for all $r$;\\
\emph{(b)} if $L$ computes $d_r$, then $h^0(L) = r+1$ and $L$ is generated;\\
\emph{(c)} $d_{r+s} \leq d_{r} + d_s$ for any $r, s \geq 1$;\\
\emph{(d)} $d_{r} \leq r(g-1)$.
\end{lem}

\begin{proof}
(a) and (b) are obvious. (c): Suppose $L$ computes $d_r$ and $M$ computes $d_s$. 
The map
$$
H^0(L) \otimes H^0(M) \ra H^0(L \otimes M)
$$
satisfies the hypotheses of the Hopf lemma. Hence 
$$
h^0(L \otimes M) \geq h^0(L) + h^0(M) -1 = r+s+1.
$$
(d) follows from (c) and the fact that $d_1 \leq g-1$, since $g \geq 4$. 
\end{proof}

\begin{lem} \label{lem4.3}
If $d_r + d_s = d_{r+s}$, then $d_n = nd_1$ for all $n \leq r+s$.
\end{lem}

\begin{proof}
Suppose $L$ computes $d_r$ and $M$ computes $d_s$. Then $h^0(L \otimes M) \geq r + s + 1$ as in the proof of Lemma \ref{lem4.1} (c).
On the other hand, $\deg (L \otimes M) = d_r + d_s = d_{r+s}$. If $h^0(L \otimes M) > r+s+1$, then $\deg(L \otimes M) \geq d_{r+s+1}$
which contradicts Lemma \ref{lem4.1} (a). So
$$
h^0(L \otimes M) = r+s+1.
$$
It now follows from \cite[Corollary 5.2]{e} (see also \cite[Lemma 1.8]{cm}) that there 
exists a line bundle $N$ with $h^0(N) \geq 2$ such that
$$
L \simeq N^r \quad \mbox{and} \quad M \simeq N^s.
$$
Hence
$$
d_r = \deg L = r \deg N \geq r d_1.
$$
By Lemma \ref{lem4.1} (c), we have $d_r = rd_1$ and similarly $d_s = sd_1$. So $d_{r+s} = (r+s)d_1$. Then 
$$
d_{r+s} = (r+s)d_1 = nd_1 + (r+s-n)d_1 \geq d_n + d_{r+s-n} \geq d_{r+s}.
$$
So the inequalities must all be equalities. In particular $d_n = nd_1$.
\end{proof}

\begin{rem} \label{rem4.2}
\emph{ (a) Clifford's theorem implies that $d_r \geq 2r$ for $r \leq g-1$; moreover $d_{g-1} = 2g-2$.\\
(b) Riemann-Roch implies that $d_r = r + g$ for $r \geq g$.\\
(c) Brill-Noether theory implies that $d_r \leq g - \left[\frac{g}{r+1} \right] + r$ for all $r$ and for a general curve we have
$$
d_r = g - \left[\frac{g}{r+1} \right] + r.
$$ So for a general curve we know the gonality sequence.
Indeed, for our purposes, we can define a general curve to be one which has this gonality sequence.}
\end{rem}

\begin{rem} \label{rem4.3}
\emph{(a) If $C$ is hyperelliptic, then 
$$
d_r = \left\{ \begin{array}{lcl}
              2r & \mbox{for} & r \leq g-1,\\
              r+g & \mbox{for} & r
              \geq g.
              \end{array}  \right. 
$$}
 
\emph{(b) If $C$ is trigonal, then 
$$
d_r =\left\{ \begin{array}{lcl}
             3r &  &  1 \leq r \leq \left[ \frac{g-1}{3} \right],\\
             r + g - 1 - \left[ \frac{g-r-1}{2} \right] & \mbox{for} & \left[ \frac{g-1}{3} \right] < r \leq g-1,\\
             r+g & & r \geq g.
             \end{array}  \right. 
$$
This follows from Maroni's theory (see \cite[Proposition 1]{ms}) and, for $r > [\frac{g-1}{3}]$, 
Serre duality and Riemann-Roch.}

\emph{(c) If $C$ is a general $k$-gonal curve, $k \geq 4$, then by \cite[Theorem 3.1]{k},
$$
d_r = kr \quad \mbox{for} \quad 1 \leq r \leq \frac{1}{k-2} \left[ \frac{g-4}{2} \right].
$$
For $k = 4$, this can be marginally improved by \cite[Theorem 4.3.2]{cm1} and then extended 
using Serre duality and Riemann-Roch
to give  
$$
d_r = \left\{ \begin{array}{lcl}
      4r &  & 1 \leq r \leq \left[ \frac{g-1}{4} \right],\\
      r+g-1 - \left[ \frac{g-r-1}{3} \right] & \mbox{for} & \left[ \frac{g-1}{4} \right] < r \leq g-1,\\
      r+g & & r \geq g,
      \end{array}  \right. 
$$
except when $g \equiv 0 \mod 4$, in which case $d_{\frac{g}{4}} = g-1$ (see \cite[Proposition 3.3]{cm}).}

\emph{(d) If $C$ is bielliptic of genus $g \geq 5$, then $C$ is quadrigonal, but its gonality sequence is 
quite different from that of (c). In fact,
$$
d_r = \left\{ \begin{array}{lcl}
              2r+2 && 1 \leq r \leq g-3,\\
              2g-3 & \mbox{for} & r=g-2,\\
              2g-2 && r=g-1,\\
              r+g && r \geq g.
\end{array}  \right. 
$$}
\end{rem}

\begin{lem} \label{lem4.4}
\hspace*{1cm} $d_r \geq \min \{ \gamma_1 + 2r, g + r -1 \}$. 
\end{lem}

\begin{proof}
Let $L$ be a line bundle computing $d_r$. Then $h^1(L) = r + g - d_r$. If $d_r < g+r-1$, 
then $L$ contributes to $\gamma_1$ and $d_r \geq \gamma_1 + 2r$.
\end{proof}

\begin{lem}
$\gamma_1 = d_r -2r$, where $r=1$, except in the following cases:\\
\emph{(a)} if $C$ is a smooth plane curve, then $r=2$; \\
\emph{(b)} if $C$ is exceptional in the sense of \cite{elms}, then $r \leq \frac{g+2}{4}$ and for $r \leq 9$ we have
$r = \frac{g+2}{4}$.
\end{lem}

\begin{proof}
This is a consequence of the results of \cite{elms}. 
\end{proof}

The following lemma is a restatement of \cite[Lemma 3.9]{pr}

\begin{lem} \label{lempr}
Let $E$ be a vector bundle of rank $n$ with $h^0(E) \geq n+s$, $s \geq 1$. Suppose that $E$ 
has no proper subbundle $N$ with $h^0(N) >\emph{\rk} N$. Then
$$
\deg E \geq d_{ns}.
$$
\end{lem}

\begin{proof}
Let $E$ be as in the statement of the lemma. Note that $h^0(E^*) = 0$, for otherwise the kernel $N$ of a non-zero 
homomorphism $E \ra \cO_C$ would contradict the hypothesis.
\cite[Lemma 3.9]{pr} now implies that $h^0(\det E) \geq ns + 1$. So $\deg E = \deg \det E \geq d_{ns}$.
\end{proof}

Suppose the line bundle $L$ computes $d_{n}$. Define a line bundle $E_L$ of rank $n$ and degree $d_{n}$ by the exact sequence
\begin{equation} \label{eq5.1}
0 \ra E_L^* \ra H^0(L) \otimes \cO_C \ra L \ra 0.
\end{equation}
As mentioned in the introduction, this is a special case of the {\it dual span construction} \cite{bu}.
Note that $\mu(E_L) = \frac{d_n}{n} \leq g-1$ by Lemma \ref{lem4.1} (d).

\begin{prop} \label{prop5.1}
\emph{(a)} $h^0(E_L^*) = 0$;\\
\emph{(b)} $E_L$ is generated;\\
\emph{(c)} if $n \geq \,(>)\; g$, then $E_L$ is semistable (stable);\\
\emph{(d)} if $\frac{d_{p}}{p} \geq \frac{d_{n}}{n}$ for all $p < n$, then $E_L$ is semistable;\\
\emph{(e)} if $\frac{d_{p}}{p} > \frac{d_{n}}{n}$ for all $p < n$, then $E_L$ is stable.
\end{prop}

\begin{proof}
(a) and (b) are obvious. (c): If $n \geq g$, then $d_{n} = n+g$ by Remark \ref{rem4.2} (b). Now apply \cite[Theorem 1.2]{b}
for the case of a line bundle.

(d): Let $M$ be a quotient bundle of $E_L$ of rank $p<n$. It follows from (a) and (b) that $M$ 
is generated with $h^0(M^*) = 0$. So $h^0(M) \geq p+1$.
Choose a $(p+1)$-dimensional subspace $V$ of $H^0(M)$ which generates $M$. Then we have an exact sequence 
$$
0 \ra \det M^* \ra V \otimes \cO_C \ra M \ra 0.
$$
Dualizing this, we see that $\det M$ is a line bundle with $h^0(\det M) \geq p+1$. So
$$
\deg M = \deg \det M \geq d_{p}.
$$
Under the hypothesis of (d), $\mu(M) \geq \mu(E_L)$. Since this is true for all quotient bundles of $E_L$, this proves that
$E_L$ is semistable. 

For (e) the proof proceeds as for (d), but now we have $\mu(M) > \mu(E_L)$ for every proper quotient bundle $M$ of $E_L$. Hence
$E_L$ is stable. 
\end{proof}

\begin{rem} \label{rem4.8}
\emph{It follows from Remark \ref{rem4.2} (a) and (b) that the hypothesis of Proposition \ref{prop5.1} (d) 
is satisfied for all 
$n \geq g$ and similarly the hypothesis of Proposition \ref{prop5.1} (e) is satisfied for all $n > g$. So these two 
parts of the proposition are generalisations of part (c). Note also that $\frac{d_{g-1}}{g-1} = \frac{d_g}{g} = 2$ 
by the same remark.
}
\end{rem} 

\begin{prop} \label{prop4.13}
Suppose $E$ is a semistable vector bundle of rank $n$ with 
$\deg E < \frac{n d_p}{p}$ for all $p \leq n$. Then $h^0(E) \leq n$.
\end{prop}

\begin{proof}
Suppose $h^0(E) \geq n+1$. If $E$ possesses no proper subbundle $N$ with $h^0(N) \geq \mbox{rk} N + 1$, then 
$\deg E \geq d_{n}$ by Lemma \ref{lempr}, contradicting the hypotheses.

So let $N$ be a subbundle of rank $p$ of $E$ minimal with respect to the property $h^0(N) \geq p+1$.
Then Lemma \ref{lempr} applies to $N$ and 
$$
\frac{\deg N }{p} \geq \frac{d_{p(h^0(N) - p)}}{p} \geq \frac{d_p}{p} > \frac{\deg E}{n},
$$
contradicting the semistability of $E$.
\end{proof}

\begin{cor} \label{lem5.3}
Suppose $\frac{d_{p}}{p} \geq \frac{d_{n}}{n}$ for $p<n$ and $E$ is a semistable vector bundle of rank $n$ with 
$\deg E < d_{n}$. Then $h^0(E) \leq n$. \hfill$\square$ 
\end{cor}

\begin{cor} \label{cor4.12}
If $d_n = nd_1$, then $\gamma_n = \gamma_n'$.
\end{cor}

\begin{proof}
If $d_n = nd_1$, then, as in the proof of Lemma \ref{lem4.3} we have $d_p = pd_1$ for all $p \leq n$. 

Now suppose that $E$ contributes to $\gamma_n$, but not to $\gamma_n'$. Then
$d \geq d_n$ by Corollary \ref{lem5.3} and $h^0(E) \leq 2n-1$. So
$$
\gamma(E) \geq \frac{1}{n}\big(nd_1 - 2(n-1)\big) = d_1 - 2 + \frac{2}{n} > \gamma_1.
$$
Hence 
$$
\gamma_n \geq \min \{ \gamma_n',\gamma_1 \} = \gamma_n'.
$$
It follows from Lemma \ref{lemma2.2} that $\gamma_n = \gamma_n'$.
\end{proof}

\begin{rem} \label{rem4.13}
\emph{The assumption $d_n = nd_1$ is valid for hyperelliptic curves of genus $g \geq n$, trigonal 
curves of genus $g \geq 3n+1$ and general quadrigonal curves of genus $g \geq 4n+1$ (see Remark \ref{rem4.3}). 
For hyperelliptic and trigonal curves we already have that $\gamma_n = \gamma_n' = \gamma_1$ by Proposition \ref{prop3.2} (a).
For quadrigonal curves we have $\gamma_n' = \gamma_1 = 2$ by Proposition \ref{prop3.7}; also $\gamma_n = 2$ for $g \geq n+4$ by 
Lemma \ref{lemma2.2} and Theorem \ref{thm3.5} (e). The corollary also applies to general 
$k$-gonal curves of genus $g \geq 2(k-2)n + 4$.}
\end{rem}

\begin{theorem} \label{pr4.9}
Let $E$ 
be a semistable bundle of rank $n$ and degree $d_n$. 

\emph{(a):} If $\frac{d_p}{p} \geq \frac{d_{p+1}}{p+1}$ for all $p < n$ and $d_n \neq nd_1$, then
$$
h^0(E) \leq n+1
$$
and there exist semistable bundles of rank $n$ and degree $d_n$ with $h^0 = n+1$.

\emph{(b):} If $d_n = nd_1$, then 
$$
h^0(E) \leq 2n
$$
and there exist semistable bundles of rank $n$ and degree $d_n$ with $h^0 = 2n$.

\emph{(c):} If $\frac{d_p}{p} \geq \frac{d_n}{n}$ for all $p < n$ and $E$ is stable, then 
$$
h^0(E) \leq n+1.
$$
\end{theorem}

\begin{proof}
(a): Suppose $h^0(E) = n+s$ with $s \geq 2$. 
If $E$ possesses no subbundle $N$ of rank $p < n$ with $h^0(N) \geq p+1$, 
then by Lemma \ref{lempr},
\begin{equation} \label{eqn3}
\deg E \geq d_{ns} > d_{n},
\end{equation}
which is a contradiction.

Now let $N$ be a subbundle of minimal rank $p$ such that $h^0(N) \geq p+1$. 
Then Lemma \ref{lempr} implies that
$$
\deg N \geq d_{p(h^0(N) -p)}.
$$
Hence
\begin{equation} \label{eqn4}
\frac{\deg N}{p} \geq \frac{d_{p(h^0(N) -p)}}{p} \geq \frac{d_{p}}{p} \geq \frac{d_{n}}{n},
\end{equation}
which contradicts the semistability of $E$ unless all these inequalities are equalities. So $h^0(N) = p+1,\; \deg N = d_{p}$
and $\frac{d_{p}}{p} = \frac{d_{n}}{n}$, i.e. $\mu(N) = \mu(E)$.

It follows that $E/N$ is semistable of rank $n-p$ and degree $d_{n} - d_{p}$. Now $d_n - d_p < d_{n-p}$ by Lemma \ref{lem4.3}.
So $h^0(E/N) \leq n-p$ by Corollary \ref{lem5.3} and $h^0(E) \leq n+1$. 

To prove existence, let $L$ be a line bundle of degree $d_n$ with $h^0(L) = n+1$. Then 
by Proposition \ref{prop5.1}, $E_L$ is a semistable bundle of rank $n$ 
and degree $d_n$ with $h^0(E_L) \geq n+1$. This completes the proof of (a).

We prove (b) by induction, the case $n=1$ being obvious. If $h^0(E) \leq n+1$ the result is clear. So suppose
$h^0(E) = n+s$ with $s \geq 2$. Arguing as in the proof of (a) we obtain a proper subbundle $N$ of $E$ of rank $p$ and degree 
$d_p$ with $h^0(N) = p+1$. Moreover, $E/N$ is semistable of rank $n-p$ and degree $d_n-d_p$, where now $d_n-d_p = d_{n-p}$.
By inductive hypothesis, we have $h^0(E/N) \leq 2(n-p)$ and hence
\begin{equation} \label{equn4.5}
h^0(E) \leq 2(n-p) + p+1 = 2n - p + 1 \leq 2n.
\end{equation}

To prove existence, choose a line bundle $L$ of degree $d_1$ with $h^0(L) = 2$ and take $E = \oplus_{i=1}^n L$.
In fact, $E$ is the dual span of the line bundle $L^n$.

(c): If $E$ is stable, then \eqref{eqn3} and \eqref{eqn4} give contradictions. Hence $h^0(E) \leq n+1$.
\end{proof}

\begin{cor} \label{cor4.16}
Suppose $\frac{d_p}{p} \geq \frac{d_{p+1}}{p+1}$ for all $p < n$ and let $E$ be a semistable bundle of rank $n$ and degree 
$d \leq d_n$. Then Conjecture \ref{conj3.1} holds for $E$.
\end{cor}

\begin{proof}
If $d < d_n$, then $h^0(E) \leq n$ by Corollary \ref{lem5.3} 
in accordance with Conjecture \ref{conj3.1} (ii). 

If $d = d_n = nd_1$, then, since $d_1 \geq \gamma_1 + 2$, we have by Theorem \ref{pr4.9},
$$
h^0(E) \leq 2n \leq \frac{d-\gamma_1n}{2} + n
$$
in accordance with Conjecture \ref{conj3.1} (i).

If $d = d_n \neq nd_1$, then $h^0(E) \leq n+1$ by Theorem \ref{pr4.9}. Now
$$
\frac{d_n - n}{\gamma_1 + 1} \geq \frac{d_1 - 1}{\gamma_1 + 1} \geq 1,
$$
since $d_1 \geq \gamma_1 + 2$. So
$$
h^0(E) \leq \frac{d_n - n}{\gamma_1 + 1} + n
$$
in accordance with Conjecture \ref{conj3.1} (ii).
\end{proof}

For general $C$, we can prove a precise result on the stability of $E_L$.  This can be deduced from \cite[Theorem 2]{bu}; for the sake of completeness and since \cite{bu} is unpublished, we include a proof. 

\begin{prop} \label{cor5.2}
Suppose $C$ is general and $L$ is a line bundle computing $d_n$. Then $E_L$ is semistable and $h^0(E_L) = n+1$.
Moreover, $E_L$ is stable unless $n=g$ and $\det E_L$ is isomorphic to $K_C(p_1 + p_2)$ for some $p_1,p_2 \in C$. 
\end{prop}

\begin{proof}
For $n > g$ stability is proved in Proposition \ref{prop5.1} (c). So suppose $n \leq g$. 
According to Proposition \ref{prop5.1}(d) 
and Remark \ref{rem4.2} (c),
in order to prove semistability it suffices to show that 
$$
\frac{1}{p} \left( g - \left[\frac{g}{p+1} \right] \right) \geq \frac{1}{n} \left( g - \left[\frac{g}{n+1}\right] \right).
$$
This is satisfied if 
$$
\frac{1}{p} \left(g - \frac{g}{p+1} \right) \geq \frac{1}{n} \left(g - \frac{g-n}{n+1} \right)
$$
which is equivalent to 
\begin{equation} \label{eq3}
\frac{g}{p+1} \geq \frac{g+1}{n+1}.
\end{equation}
This is true for $n \leq g$. For $n < g$ the same proof shows that \eqref{eq3} is true with strict inequality. 
Since $d_n \neq nd_1$, Theorem \ref{pr4.9} shows that $h^0(E_L) = n+1$.

When $n=g$ and $\det E_L \not \simeq K_C(p_1+p_2)$, it follows directly from 
\cite[Theorem 1.2]{b} that $E_L$ is stable.
\end{proof}

\begin{rem} \label{rem4.11}
\emph{Suppose $C$ is hyperelliptic and $n \leq g-1$. Then by Remark \ref{rem4.3} (a), $d_n = 2n$. If $H$ is the 
hyperelliptic line bundle, then 
$$
E = \bigoplus_{i=1}^n H
$$
is semistable of degree $d_n$ with $h^0(E) = 2n$. Moreover, $E$ is generated, so we can choose a subspace $V$ of $H^0(E)$
of dimension $n+1$ which generates $E$, giving an exact sequence
$$
0 \ra \det E^* \ra V \otimes \cO_C \ra E \ra 0.
$$
Noting that $\det E \simeq H^n$ and dualizing, we obtain
\begin{equation} \label{eq4.6}
0 \ra E^* \ra V^* \otimes \cO_C \ra H^n \ra 0.
\end{equation}
Now $\deg H^n = d_n$. So $h^0(H^n) = n+1$ and \eqref{eq4.6} is the evaluation sequence of $H^n$. Thus 
$$
E \simeq E_{H^n}.
$$
Moreover $H^n$ is the unique line bundle of degree $d_n$ with $h^0 \geq n+1$. 
So if $2 \leq n \leq g-1$, the bundle $E_L$ constructed in \eqref{eq5.1} can never be stable.
}
\end{rem}

\begin{rem} \label{rem4.12}
\emph{If $C$ is a trigonal curve and $n \leq \left[ \frac{g-1}{3} \right]$, 
then $d_n = 3n$ by Remark \ref{rem4.3} (b). 
So we can use a similar argument to that of Remark \ref{rem4.11}, by replacing $H$ by the unique 
line bundle $T$ of degree 3 computing 
$d_1$. Then
$$
E = \bigoplus_{i=1}^n T
$$
is semistable of degree $d_n$ with $h^0(E) = 2n$.
Moreover, if $n < \frac{g-1}{3}$, then $T^n$ is the unique line bundle of degree $d_n$ 
with $h^0 \geq n+1$ by \cite[Proposition 1]{ms}. So if $2 \leq n < \frac{g-1}{3}$, then 
the bundle $E_L$ constructed in \eqref{eq5.1} can never be stable. Note that we need $g \geq 8$ in order to allow $n \geq 2$.
}
\end{rem}

\begin{rem} 
\emph{Similarly, if $C$ is a general $k$-gonal curve ($k \geq 4$) of genus $g \geq 4k-4 \;(g \geq 9$ if $k=4)$ 
and $n \leq \frac{1}{k-2} \left[ \frac{g-4}{2} \right]$ $(n \leq [\frac{g-1}{4} ]$ if $k = 4)$, then by
Remark \ref{rem4.3} (c) we have $d_n = kn$. Let $Q$ be a line bundle of degree $k$ with $h^0(Q) = 2$. Then
$$
E = \bigoplus_{i=1}^n Q
$$
is semistable of degree $d_n$ with $h^0(E) = 2n$. For $n \geq 2$, $E$ is strictly semistable.} 

\emph{Note that $Q$ is unique for $g \geq (k-1)^2 + 1$. This follows 
from the fact that a curve of type $(a,b)$ on a smooth quadric surface is of arithmetic genus $(a-1)(b-1)$.}

\emph{When $k=4$ and $g \geq 11$, $Q^2$ is the unique line bundle of degree 8 with $h^0 \geq 3$ by \cite[Theorem 3.2]{cm}. 
So for a general quadrigonal curve of genus $g \geq 11$, there do not exist stable bundles of the form $E_L$ with $L$ 
of degree 8 and $h^0(L) = 3$.
We do not know whether in other cases $Q^n$
is the unique line bundle of degree $d_n$ with $h^0 \geq n+1$. So it is possible that there could exist stable bundles of 
the form $E_L$.}
\end{rem}

The next theorem improves the results of Theorem \ref{thm3.5}.

\begin{theorem} \label{thm4.14}
Let $C$ be a curve with Clifford index $\gamma_1 \geq 2$. \\
\emph{(a):} If $n=g$, then 
$$
\gamma_n = 2 - \frac{2}{g};
$$
\emph{(b):} if $n = g-2$, then
$$
\gamma_n = 2 - \frac{1}{g-2};
$$ 
\emph{(c):} if $n=g-3$, then
$$
\gamma_n = 2.
$$
\end{theorem}
\begin{proof}
Note first that $d_g = 2g,\; d_{g-1} = 2g-2, \; d_{g-2} = 2g-3$ for  $\gamma_1 \geq 1$ and $d_{g-3} = 2g-4$ for $\gamma_1 \geq 2$.

(a): According to Theorem \ref{thm3.5} (b) we need only show that $\gamma_n \geq \frac{8}{5}$ if $n = g = 5$. This will
follow as in the proof of Theorem \ref{thm3.5} (b) if we can show that there is no semistable bundle $E$ of 
rank 5 and degree $2n+1 = 11$ with $h^0(E) \geq 7$.

Any such bundle is necessarily stable and if $F$ is an elementary transformation of $E$, then $\deg F = 10$ and $F$ is
semistable. Now $d_5 = 10 \neq 5d_1$ and one can easily check that the conditions of Theorem \ref{pr4.9} (a) hold. So $h^0(F) \leq 6$. 
Since this holds for any elementary transformation of $E$, it follows that $E$ is generated with $h^0(E) = 7$. Now, 
following through the proof of Theorem \ref{pr4.9}, we see that \eqref{eqn3} gives $\deg E \geq d_{10}$, a contradiction. 
So there must exist a subbundle $N$ of $E$ of rank $p < n$ which is minimal with respect to the condition $h^0(N) > \rk N$.
In this case we have $\deg N \geq d_p$ and, by stability of $E$, $\frac{\deg N}{p} < \mu(E) = \frac{11}{5}$. 
The only possibility for this is $p = 4, \; \deg N = d_4 = 8$. It is easy to check that $N$ must be semistable. 
So Theorem \ref{pr4.9} (a) gives $h^0(N) \leq 5$. It follows that $h^0(E/N) \geq 2$. So $\deg (E/N) \geq d_1 = 4$. Hence 
$\deg E \geq 12$, a contradiction.

(b): For $n= g-2$, we need to show that there exists a semistable bundle $E$ of rank $g-2$ and degree $2g-3$ with $h^0(E) \geq g-1$
(see the proof of Theorem \ref{thm3.5} and in particular the inequality \eqref{eq3.5}). Since 
$\frac{d_{g-2}}{g-2} = \frac{2g-3}{g-2} = 2 + \frac{1}{g-2}$ and $\frac{d_p}{p} > 2$ for all $p < g-2$ by Lemma \ref{lem4.4},
the hypotheses of Proposition \ref{prop5.1} (d) hold and we can take $E = E_L$ with $L$ a line bundle of degree $d_{g-2}$
with $h^0(L) = g-1$.   

(c): For $n = g-3$, we again consider the proof of Theorem \ref{thm3.5}. We need to show that there is no semistable bundle 
$E$ of rank $g-3$ and degree $2n+1 = 2g-5$ with $h^0(E) \geq g-2$. Since $\gamma_1 \geq 2$, Lemma \ref{lem4.4} implies that
$d_p \geq 2p+2$ for all $p < g-3$. So $\frac{d_p}{p} > \frac{d_{g-3}}{g-3} = 2 + \frac{2}{g-3}$. Hence the conditions of Corollary
\ref{lem5.3} apply, giving $h^0(E) \leq g-3$. 
\end{proof}

\begin{cor} \label{cor4.15}
If $\gamma_1 = 2$, then for all $n \geq 1$,
$$
\gamma_n = \left\{ \begin{array}{ccc}
                   1+ \frac{g-2}{n} & & n>g,\\
                   2 - \frac{2}{g} && n = g,\\
                   2-\frac{2}{g-1} & \text{if} & n= g-1,\\
                   2 - \frac{1}{g-2} && n = g-2,\\
                   2 && n \leq g-3.
                   \end{array}  \right.
$$
In particular $\gamma_2 = 2$. 
\end{cor}    
\begin{proof}
This follows from Theorems \ref{thm3.5} and \ref{thm4.14} and Lemma \ref{lemma2.2}. For the last part, note that $\gamma_1 = 2$ 
implies that $g \geq 5$; so $2 \leq g-3$.
\end{proof}

\begin{rem} \label{rem4.25}
\emph{Corollary \ref{cor4.15} applies in particular to any quadrigonal curve. If $C$ is bielliptic, rather more is known. 
By \cite{ba} (see also \cite[Theorem 2 and Lemma 3]{m}), there exist semistable bundles $E$ of any rank $n$ and degree $d \geq 2$
with $h^0(E) = [\frac{d}{2}]$; in other words,
$$
\gamma(E) = \frac{1}{n} \left(d - 2 \left[ \frac{d}{2} \right] + 2n \right) = \left\{ \begin{array}{ccc}
                                                                                     2 & \mbox{for} & d \; \mbox{even},\\
                                                                                     2 + \frac{1}{n} & \mbox{for} & d \; \mbox{odd}.
                                                                                     \end{array}  \right.
$$ 
If $2(n+1) \leq d \leq n(g-1)$, these bundles contribute to $\gamma_n$ and, if $4n \leq d \leq n(g-1)$, they contribute to 
$\gamma_n'$. Since $\gamma_n' = 2$ by Lemma \ref{lemma2.2} and Proposition \ref{prop3.7}, there are many bundles on $C$ 
which compute $\gamma_n'$.}  
\end{rem}

The remaining results of this section will be useful in estimating the value of $\gamma_n$.              

\begin{lem} \label{lem4.10}
Suppose $p|n$ and $\frac{d_q}{q} \geq \frac{d_p}{p}$ for $q < p$. Then
$$
\gamma_n \leq \frac{1}{p}(d_p - 2).
$$
\end{lem}
\begin{proof}
Let $F$ be a semistable bundle of rank $p$ and degree $d_p$ with $h^0(F) \geq p+1$, which exists by Proposition \ref{prop5.1}. 
Then, if $$E = \bigoplus_{i=1}^{\frac{n}{p}}F,$$ we have
$$
\gamma(E) = \gamma(F) \leq \frac{1}{p}(d_p - 2).
$$
The bundle $E$ contributes to $\gamma_n$, which gives the result. 
\end{proof}

\begin{prop} \label{prop4.11}
Suppose $E$ is a semistable bundle of rank $n \geq 2$.

\emph{(a)} If $\frac{d_p}{p} \geq \frac{d_n}{n}$ for all $p < n$ and $h^0(E) = n+ 1$, then 
$$
\gamma(E) \geq \frac{1}{n}(d_n -2).
$$

\emph{(b)} If $h^0(E) \geq n+2$ and there exists no proper subbundle $N$ of $E$ with $h^0(N) > \emph{\rk} N$, then
$$
\gamma(E) \geq \frac{1}{n}(d_n - 2).
$$
If $n \geq 3$, this is a strict inequality.

\emph{(c)} If $h^0(E) \leq 2n -1$ and there exists a line subbundle $N$ of $E$ with $h^0(N) \geq 2$, then
$$
\gamma(E) > \gamma_1.
$$

\emph{(d)} Suppose $h^0(E) = n+s$ with $s \geq 1$ and there exists a subbundle $N$ of $E$ of rank $p \geq 2$
with $h^0(N) =p+t$, $t \geq 1 + \frac{2s}{n} - \frac{2}{p}$, and no subbundle $N'$ of rank $<p$ with $h^0(N') > \emph{rk} N'$. Then
$$
\gamma(E) \geq \frac{1}{p} (d_p -2).
$$
If further $\frac{d_p}{p} \geq \frac{d_n}{n}$ and $t \geq 1 + \frac{2}{n} (s-1)$, then
$$
\gamma(E) \geq \frac{1}{n} (d_n -2).
$$
\end{prop}

\begin{proof}
(a): If $h^0(E) = n+1$, then by Corollary \ref{lem5.3}, $d \geq d_n$ which implies the assertion.

(b): Suppose $h^0(E) = n+s$ with $s \geq 2$. By Lemma \ref{lempr}, $\deg E \geq d_{ns}$ and
\begin{eqnarray*}
\gamma(E) &\geq &\frac{1}{n}(d_{ns} - 2s) \\
& \geq &\frac{1}{n}(d_n + ns -n - 2s)\\
& = & \frac{1}{n}\big(d_n +(n-2)s -n\big) \geq \frac{1}{n}(d_n - 2)
\end{eqnarray*}
and the last inequality is strict if $n \geq 3$.

(c):  By definition of $d_1$, $\deg N \geq d_1$. 
So by semistability, $\deg E \geq nd_1$ and hence
$$
\gamma(E) \geq \frac{1}{n}\big(nd_1 - 2(h^0(E)-n)\big) = d_1 - \frac{2(h^0(E)-n)}{n} > d_1 - 2 \geq \gamma_1.
$$

(d): Lemma \ref{lempr} gives $\deg N \geq d_{pt}$. By semistability,
$$
\deg E \geq \frac{n}{p} d_{pt} \geq \frac{n}{p}(d_p + pt -p).
$$
So 
$$
\gamma(E) \geq \frac{1}{p} (d_p +pt -p) - \frac{2s}{n}.
$$
The inequality $t \geq 1 + \frac{2s}{n} - \frac{2}{p}$ now gives 
$$
\gamma(E) \geq \frac{1}{p} (d_p -2).
$$
If $\frac{d_p}{p} \geq \frac{d_n}{n}$ and $t \geq 1 + \frac{2}{n}(s-1)$, we get
$$ \gamma(E) \geq \frac{d_p}{p} - \frac{2}{n} \geq \frac{d_n}{n} - \frac{2}{n},
$$
proving the second assertion.
\end{proof}

\begin{prop} \label{prop4.18}
Suppose $E$ is a semistable bundle of rank $n \geq 2$.
If $\frac{d_{n-1}}{n-1} \geq \frac{d_n}{n}$, $h^0(E) = n+2$ and there exists a subbundle $N$ of 
rank $n-1$ of $E$ with $h^0(N) = n$ and no 
subbundle $N'$ of $E$ of smaller rank with $h^0(N') > \emph{\rk} N'$, then
$$
\gamma(E) \geq \min \left\{ \gamma_1, \frac{1}{n}(d_n - 2) \right\}.
$$
\end{prop}

\begin{proof}
By Lemma \ref{lempr}, $\deg N \geq d_{n-1}$. 
On the other hand, the definition of $d_1$ implies that 
$\deg (E/N) \geq d_1$. Hence
$$
\deg E \geq d_{n-1} + d_1.
$$
If $d_{n-1} \geq (n-1)d_1 - 2(n-2)$, then $\deg E \geq n d_1 - 2(n-2)$ and hence
$$
\gamma(E) \geq \frac{1}{n}\big(nd_1 - 2(n-2) - 4\big) = d_1 - 2 \geq \gamma_1.
$$
This covers in particular the case $n=2$. If $n \geq 3$ and  
$$
d_{n-1} \leq (n-1)d_1 - 2(n-2) - 1,
$$ 
then
$$
d_n \leq \frac{n}{n-1} d_{n-1} = d_{n-1} + \frac{1}{n-1} d_{n-1} \leq d_{n-1} + d_1 - \frac{2n-3}{n-1},
$$
implying $d_n \leq d_{n-1} + d_1 -2$. So $\deg E \geq d_n + 2$ and hence 
$$
\gamma(E) \geq \frac{1}{n} (d_n - 2).
$$
\end{proof}

\section{Rank two}

The results of Section 4 allow us to obtain a precise formula for $\gamma_2$ and to improve the inequality for  $\gamma_2'$ obtained in Proposition \ref{prop3.6}. 

\begin{theorem} \label{prop4.7}
$\quad \gamma_2 = \min \left\{ \gamma_2', \frac{d_2}{2} - 1\right\}$.
\end{theorem}

\begin{proof}
The inequality $\gamma_2 \geq  \min \left\{ \gamma_2', \frac{d_2}{2} - 1\right\}$ is an immediate consequence of 
Proposition \ref{prop4.11} (a).
Moreover, by the definition of $\gamma_2'$ and Proposition \ref{prop5.1}, both inequalities can be equalities.
\end{proof}

\begin{theorem} \label{prop4.8}
$\quad \gamma_2' \geq \min \left\{ \gamma_1, \frac{d_4}{2} -2 \right\}$.
\end{theorem}

\begin{proof}
Suppose $E$ contributes to $\gamma_2'$. Since $E$ is semistable, any line subbundle $L$ of $E$ has $\deg L \leq g-1$. 
If $h^0(L) \geq 2$, then $\gamma(L) \geq \gamma_1$.

Now consider the quotient $L' = E/L$. By semistability of $E$, $\deg L' \geq \deg L$.
If $h^0(L') \leq h^0(L)$ then $\gamma(L') \geq \gamma(L) \geq \gamma_1$.
Note also that 
$$
\deg(K_C \otimes L'^*) = 2g-2 -\deg E + \deg L \geq \deg L.
$$
So if $h^0(K_C \otimes L'^*) \leq h^0(L)$ we have again $\gamma(K_C \otimes L'^*) \geq \gamma_1$, i.e. $\gamma(L') \geq \gamma_1$.
Otherwise we have $h^0(L') \geq 2$ and $h^1(L') \geq 2$. Hence also $\gamma(L') \geq \gamma_1$. 
This implies 
\begin{equation} \label{eq1}
\gamma(E) \geq \frac{1}{2}\big(\gamma(L) + \gamma(L')\big) \geq \gamma_1.
\end{equation}

Now suppose $E$ has no line subbundle $L$ with $h^0(L) \geq 2$. Write $h^0(E) = 2 + s$ with $s \geq 2$. 
Then by Lemma \ref{lempr}, 
\begin{equation*}
\deg E \geq d_{2s}.
\end{equation*}
By Lemma \ref{lem4.1} (a) this implies 
$$
\deg E \geq d_4 + 2s - 4.
$$
So
\begin{equation} \label{eq2}
\gamma(E) \geq \frac{1}{2}(d_4 + 2s-4 - 2s) = \frac{d_4}{2} - 2.
\end{equation}
The proposition follows from \eqref{eq1} and \eqref{eq2}.
\end{proof}

\begin{cor} \label{cor4.9}
$
\quad \gamma_2 = \min\left\{ \gamma_1, \frac{d_2}{2} - 1 \right\} \geq \min \left\{ \gamma_1, \frac{\gamma_1}{2} + 1 \right\}.
$
In particular, if $\gamma_1\ge2$ and $d_2$ computes the Clifford index, then
$$
\gamma_2 = \frac{\gamma_1}{2} + 1.
$$
\end{cor}

\begin{proof}
By Lemma \ref{lem4.1} (a), $d_4 \geq d_2 + 2$. Hence the first equality follows from 
Theorems \ref{prop4.8}, \ref{prop4.7} and Lemma \ref{lemma2.2}. The inequality follows from the fact that 
$d_2 \geq \gamma_1 + 4$. Moreover, if $d_2$ computes the Clifford index, then $d_2 = \gamma_1 + 4$.
\end{proof}

\begin{cor} \label{cor4.10}
Let $C$ be a general curve of genus $g$.\\
\emph{(a)} If $g=4$, then 
$$
\gamma_2 = \gamma_2' = \gamma_1 = 1 < \frac{d_2}{2} - 1.
$$
\emph{(b)} If $g \geq 5$, then 
$$
\gamma_2 = \frac{d_2}{2} - 1 = \frac{1}{2}\left( g - \left[ \frac{g}{3} \right] \right).
$$
\emph{(c)} If $g \geq 7, \; g \neq 8$, then 
$$
\gamma_2 < \gamma_2'.
$$
\end{cor}
\begin{proof}
When $g=4$, we note that $\gamma_1 = 1$ and $d_2 = 5$ which proves (a).

(b): By Corollary \ref{cor4.9} we have $\gamma_2 = \frac{d_2}{2} - 1$ whenever
$$
\frac{d_2}{2} - 1 \leq \gamma_1 = \left[ \frac{g-1}{2} \right].
$$
By Remark \ref{rem4.2} (c) this inequality is equivalent to 
$$
\left[ \frac{g-1}{2} \right] \geq \frac{1}{2} \left( g - \left[ \frac{g}{3} \right] \right).
$$
It is easy to see that this is true for $g \geq 5$. 

(c): By (b), $\gamma_2 = \frac{1}{2}(g - [\frac{g}{3}])$. By Theorem \ref{prop4.8},
$$
\gamma_2' \geq \min \left\{ \gamma_1, \frac{d_4}{2} - 2 \right\} = \min \left\{ \left[ \frac{g-1}{2} \right], 
\frac{1}{2} \left( g - \left[\frac{g}{5} \right] \right)  \right\}.
$$
A simple computation gives the assertion.
\end{proof}

\begin{rem} \label{rem5.5}
\emph{From Theorem \ref{prop4.8}, $\gamma_2' \geq \min \{ \gamma_1, \frac{d_4}{2} - 2 \}$. For a general curve this implies
$$
\gamma_2' \geq \min \left\{ \left[ \frac{g-1}{2} \right], \frac{1}{2} \left( g - \left[ \frac{g}{5} \right] \right) \right\}.
$$
From this and Lemma \ref{lemma2.2} we obtain $\gamma_2' = \gamma_1$ provided
$$
\left[ \frac{g-1}{2} \right] \leq \frac{1}{2} \left( g - \left[ \frac{g}{5} \right] \right).
$$
This holds for $g \leq 10, \;g= 12$ and $g = 14$. For $g \leq 10$ the fact that $\gamma_2' = \gamma_1$ 
follows already from Proposition \ref{prop3.6}. These results and Corollary \ref{cor4.10} can be deduced also from \cite[Lemma 6]{m}.}
\end{rem}

Following on from this remark, we have,

\begin{prop} \label{prop5.6}
Suppose that $\gamma_1 \geq 2$ and there is no semistable bundle on $C$ of rank 2 and degree $d$ with $h^0 \geq 2 + s,\; s \geq 2$ 
and $d_{2s} \leq d < 2 \gamma_1 + 2s$. Then Conjecture \ref{conj3.1} holds for $n = 2$.
\end{prop}

\begin{proof}
If $E$ contributes to $\gamma_2'$ and $\gamma(E) < \gamma_1$, the proof of Theorem \ref{prop4.8} (see in particular \eqref{eq1}) shows that $d = \deg E \geq d_{2s}$  where $s = h^0(E) - 2$. The hypotheses now imply that $d \geq 2 \gamma_1 + 2s$.
So $\gamma(E) \geq \frac{1}{2}(2 \gamma_1 + 2s - 2s) = \gamma_1$, a contradiction. It follows that $\gamma_2' = \gamma_1$. 
Moreover, $E$ is in the range of Conjecture \ref{conj3.1} (i).

Suppose now that $E$ contributes to $\gamma_2$, but not to $\gamma_2'$. Then $h^0(E) = 3$ and by Corollary \ref{lem5.3}, 
$d \geq d_2$. Since $d_2 \geq \gamma_1 + 4$ by Lemma \ref{lem4.4}, this gives 
$$
\frac{1}{\gamma_1 + 1}(d-2) + 2 \geq \frac{\gamma_1 + 2}{\gamma_1 + 1} + 2 > h^0(E),
$$
verifying Conjecture \ref{conj3.1} (ii).
\end{proof}

\begin{rem} \label{rem5.7}
\emph{The hypotheses of Proposition \ref{prop5.6} certainly hold if $d_4 \geq 2 \gamma_1 + 4$. In any case, a semistable bundle $E$ on $C$ of rank 2 and degree $d\le2g-2$ 
with $h^0(E) \geq 4$ and $\gamma(E) < \gamma_1$ has no line subbundle with $h^0 \geq 2$ (compare again the proof of Theorem \ref{prop4.8}). We could consider extensions of the form 
$$
0 \ra L \ra E \ra M \ra 0,
$$
where $L$ and $M$ are line bundles with $h^0(L) = 1$ and $h^0(M) \geq s+1$. The problem  in constructing $E$ in this way is that one needs to lift a 
large number of sections of $M$ to $E$; this is a difficult problem and is likely to require geometric information about 
$C$ beyond that provided by the gonality sequence (compare \cite{v}).} 
\end{rem}

\section{Ranks three and four}

\begin{theorem} \label{prop5.4}
Suppose $\frac{d_2}{2} \geq \frac{d_3}{3}$. Then
$$
\gamma_3 = \min \left\{ \gamma_3', \frac{1}{3}(d_3 - 2) \right\}.
$$
\end{theorem}
\begin{proof}
Suppose $E$ contributes to $\gamma_3$. 
If $h^0(E) \geq 6$, then $\gamma(E) \geq \gamma_3'$ by definition of $\gamma_3'$. All other possibilities are covered 
by Propositions \ref{prop4.11} and \ref{prop4.18}.
Hence $\gamma_3 \geq \min \{ \gamma_3', \frac{1}{3}(d_3 - 2) \}.$

Moreover, by the definition of $\gamma_3'$ there exists $E$ with $\gamma(E) = \gamma_3'$. By Proposition \ref{prop5.1} (d), there
exists a semistable bundle $E$ of rank 3 and degree $d_3$ with $h^0(E) \geq 4$. Since $d_3 \leq 3g-3$ by Lemma \ref{lem4.1} (d), 
this gives $\gamma_3 \leq \gamma(E) \leq \frac{1}{3} (d_3 - 2)$. This completes the proof.
\end{proof}

\begin{theorem} \label{prop7.1}
If $ \frac{d_3}{3} \geq \frac{d_4}{4}$, then
$$
\gamma_4 = \min \left\{ \gamma_4', \frac{1}{4}(d_4 -2), \frac{1}{2}(d_2-2) \right\}.
$$ 
\end{theorem} 

\begin{proof}
The result holds for $\gamma_1 = 0$ and 1 by Proposition \ref{prop3.2} (a) and the fact that $d_2 \geq 4$ and $d_4 \geq 6$. 
So suppose $\gamma_1 \geq 2$.
 
Suppose that $E$ contributes to $\gamma_4$. In order to prove the inequality 
$$
\gamma_4 \geq \min \left\{ \gamma_4', \frac{1}{4}(d_4 -2), \frac{1}{2}(d_2-2) \right\}.
$$
we may assume by Propositions \ref{prop4.11} and \ref{prop4.18} that
$h^0(E) = 7$ and $E$ admits a subbundle $N$ of rank $p$ with $2 \leq p \leq 3$ and $h^0(N) = p+1$ and such that $E$ 
does not admit a subbundle of smaller rank with $h^0 > \rk$. 

{\it Case 1:} $p=2$. We have $h^0(N) = 3$ and $h^0(E) = 7$, so $h^0(E/N) \geq 4$. If $E/N$ has no line subbundle with $h^0 \geq 2$,
then Lemma \ref{lempr} gives 
$$
\deg E \geq \deg N + \deg (E/N) \geq d_2 +d_4 \geq d_4 + 4.
$$ 
This implies
\begin{equation} \label{eq7.2}
\gamma(E) \geq \frac{1}{4}(d_4 + 4 - 6) = \frac{1}{4}(d_4 -2).
\end{equation}

If $E/N$ has a line subbundle $M$ with $h^0(M) \geq 3$, then $\deg M \geq d_2$ by definition of $d_2$ and 
$\deg ((E/N)/M ) \geq \frac{d}{4} \geq \frac{d_2}{2}$, since $E$ is semistable. So 
$$
\deg E \geq d_2 + d_2 + \frac{d_2}{2} \geq 2d_2 + 2,
$$
since $d_2 \geq 4$. Hence
\begin{equation} \label{eq6.2}
\gamma(E) \geq \frac{1}{4} ( 2d_2 + 2 - 6) = \frac{1}{2}(d_2 - 2).
\end{equation}

If $E/N$ has a line subbundle $M$ with $h^0(M) = 2$, then $\deg M \geq d_1$ and $\deg ((E/N)/M) \geq d_1$. So
$$
\deg E \geq d_2 + 2d_1.
$$
If $d_2 \leq 2d_1 -2$, then $\deg E \geq 2d_2 + 2$, so \eqref{eq6.2} holds.
If $d_2 \geq 2d_1 - 1$, then $\deg E \geq 4 d_1 -1$ and hence
\begin{equation} \label{eq6.3}
\gamma(E) \geq \frac{1}{4} ( 4 d_1 - 1 - 6) = d_1 - \frac{7}{4} > d_1 - 2 \geq \gamma_1 \geq \gamma_4'.
\end{equation}

{\it Case} 2: $p=3$. We have $h^0(N) = 4$ and $h^0(E) = 7$.
So $h^0(E/N) \geq 3$ and hence $�\deg (E/N) \geq d_2$ by definition of $d_2$. 
If $d_3 \geq d_2 +2$, then $\deg E \geq 2d_2 + 2$ and hence again \eqref{eq6.2} holds.
If $d_3 = d_2 + 1$, then using the hypothesis, 
$$ 
\deg E \geq  d_2 + d_3 =2d_3 -1 \geq \frac{3}{2}d_4 -1. 
$$
If $d_4 \geq 10$, this gives $\deg E \geq d_4 +4$, so \eqref{eq7.2} holds. Otherwise, by Lemma \ref{lem4.4},
we have $g=5$ or $g=6$. In either case, using Remark \ref{rem4.2} (c) and Lemma \ref{lem4.4}, we get $d_2 = 6$ and $d_4 \leq 9$. So
$$
\deg E \geq d_2 + d_3 \geq d_4 + 4
$$
and again \eqref{eq7.2} holds.

The inequality $\gamma_4 \geq \min \left\{ \gamma_4', \frac{1}{4}(d_4 -2), \frac{1}{2}(d_2-2) \right\}$ 
follows from the inequalities \eqref{eq7.2}, \eqref{eq6.2} and \eqref{eq6.3}.

The proof of the equality is similar to the last part of the proof of Theorem \ref{prop5.4}.
To obtain $\gamma(E) = \frac{1}{2}(d_2-2)$, we define $E$ to be $E = E_L \oplus E_L$ where $L$ is a 
line bundle of degree $d_2$ with $h^0(L) = 3$.
\end{proof}

\begin{rem}
\emph{The conditions $\frac{d_2}{2} \geq \frac{d_3}{3}$ and $\frac{d_3}{3} \geq \frac{d_4}{4 }$
are satisfied for general curves (see proof of Proposition \ref{cor5.2}) 
and also for hyperelliptic, trigonal, general quadrigonal and bielliptic curves (see Remark \ref{rem4.3}).}

\emph{For $\gamma_1 \leq 2$, we already know that $\gamma_n' = \gamma_1$ and we have precise values for the $\gamma_n$ 
from Corollary \ref{cor4.15}, so Theorems \ref{prop5.4} and \ref{prop7.1} do not add anything to our knowledge in these
cases. For a general curve of genus $g \geq 7$, we can show that $\frac{1}{2}(d_2 - 2), \frac{1}{3}(d_3 - 2)$ and
$ \frac{1}{4}(d_4 - 2)$ are all $\geq 2$. So, in the absence of any good lower bound for $\gamma_n'$, 
Theorems \ref{prop5.4} and \ref{prop7.1} tell us that 
$$2 \leq \gamma_n \leq \gamma_n' \leq \gamma_1
$$
for $n=3,4$.
} 
\end{rem}

\section{Rank five}

In this section we obtain partial results for $\gamma_5$. 

\begin{lem} \label{lem7.1}
Suppose $\frac{d_4}{4} \geq \frac{d_5}{5}$. Let $E$ be a semistable bundle of rank $5$ with $h^0(E) \leq 9$ 
and $N$ a subbundle of rank $p$, $2 \leq p \leq 4$, with $h^0(N) \geq p+2$. Suppose further that $E$ has no subbundle of 
rank $< p$ with $h^0 > \emph{\rk}$. Then
$$
\gamma(E) \geq \frac{1}{5}(d_5 - 2).
$$ 
\end{lem}
\begin{proof}
Lemma \ref{lempr} implies that $\deg N \geq d_{2p}$. Semistability gives 
$$
\deg E \geq \frac{5d_{2p}}{p}.
$$ 
If $p = 2$, the hypothesis implies that
$$
\deg E \geq \frac{5d_4}{2} \geq \frac{5}{2} \frac{4}{5} d_5 = 2 d_5 \geq d_5 + 8,
$$
since $d_5 \geq 8$ by Lemma \ref{lem4.4}. If $p = 3$ or $4$, then $2p > 5$, so
$$
\deg E \geq \frac{5d_{2p}}{p} \geq \frac{5}{p}(d_5 + 2p-5) = d_5 + \frac{5-p}{p}d_5 +10 -\frac{25}{p} \geq d_5 + 2 + \frac{15}{p}.
$$
So $\deg E \geq d_5 + 6$. Hence in any case
$$
\gamma(E) \geq \frac{1}{5}(d_5 - 2).
$$
\end{proof}

\begin{lem} \label{lem7.2}
Let $F$ be a vector bundle of rank $2$ with $h^0(F) \geq 2 + t,\; t \geq 1$. Suppose $F$ is a quotient of a 
semistable bundle $E$ of rank $n$ and degree $d > 0$. Then
$$
\deg F \geq \min_{1 \leq u \leq t-1}  \left\{ d_{2t}, d_t + \frac{d}{n}, d_u + d_{t-u} \right\}.
$$
\end{lem}

\begin{proof}
If $F$ has no line subbundle with $h^0 \geq 2$, then Lemma \ref{lempr} implies that 
$\deg F \geq d_{2t}$. 

Otherwise let $N$ be a line subbundle with $h^0(N) = 1 + u, \; u \geq 1$.
If $u \geq t$, then $\deg N \geq d_t$ and $\deg (F/N) \geq \frac{d}{n}$ by semistability of $E$.
If $u \leq t-1$, then $\deg N \geq d_u$ and $\deg (F/N) \geq d_{t-u}$, since $h^0(F/N) \geq t - u + 1 \geq 2$.
\end{proof}

\begin{theorem} \label{prop8.1}
If $ \frac{d_p}{p} \geq \frac{d_{p+1}}{p+1}$ for $1 \leq p \leq 4$, then
\begin{eqnarray*}
\gamma_5 &\geq  \min &\left\{  \gamma_5', \frac{1}{2}(d_2 -2), \frac{1}{5}(d_5 -2),  \frac{1}{5}(d_1 + 2d_2 - 6), 
\frac{1}{5}(d_1 + d_4 -4), \right.\\
&& \left. \hspace*{1.2cm} \frac{1}{5}(2d_1 + d_3 - 6), \frac{1}{5}(3d_1 + d_2 -8) , \frac{1}{5}(d_2 + d_3 -5) \right\}. 
\end{eqnarray*}
\end{theorem} 

\begin{proof}
Let $E$ be a semistable vector bundle of rank 5 and degree $d$. By Proposition \ref{prop4.11} and Lemma \ref{lem7.1}
we may assume that $h^0(E) = 5+s$ with $2 \leq s \leq 4$ and $E$ admits a subbundle $N$ of rank $p$ 
with $2 \leq p \leq 4$ and $h^0(N) = p+1$ and such that $E$ 
does not admit a subbundle of smaller rank with $h^0 > \rk$.
 
If $\gamma_1 \leq 2$ or $ g \leq 8$, we have precise values for all $\gamma_n$ by Proposition \ref{prop3.2} (a) and Theorems \ref{thm3.5} and \ref{thm4.14}. One can check that these values satisfy the required inequality.
So we may assume that $\gamma_1 \geq 3$ and $g \geq 9$, implying by Lemma \ref{lem4.4} that 
$$
d_2 \geq 7, \; d_3 \geq 9, \; d_4 \geq 11 \; \mbox{and} \; d_5 \geq 13.
$$ 
In fact, in the proof we use only $\gamma_1 \geq 2, \; d_2 \geq 6$ and $d_3 \geq 9$.

{\it Case} 1: $p=2$.
We have $h^0(N) = 3$. So Lemma \ref{lempr} implies $\deg N \geq d_2$ and hence $\frac{d}{5} \geq \frac{d_2}{2} \geq 3$ 
by semistability of $E$.
Moreover, $E/N$ has rank 3 and $h^0(E/N) \geq s + 2$.

{\it Case} 1 a: Suppose $E/N$ has no proper subbundle with $h^0 > \rk$. Then by Lemma \ref{lempr}, $\deg (E/N) \geq d_{3(s-1)}$. So
$$
\deg E \geq d_2 + d_{3(s-1)} \geq d_2 + d_3 + 3(s-2)$$
and
\begin{equation} \label{eq8.1}
\gamma(E) \geq \frac{1}{5}\big(d_2 + d_3 +3(s-2) -2s\big) \geq \frac{1}{5}(d_2 + d_3 -4).
\end{equation}

{\it Case} 1 b: Suppose $E/N$ has a line subbundle $M$ with $h^0(M) \geq 2$. 
If $h^0(M) \geq 3$, then $\deg M \geq d_2$. So by semistability of $E$,
$$
\frac{d}{5} \leq \frac{\deg((E/N)/M)}{2} \leq \frac{d-2d_2}{2}.
$$
Hence 
$$
d \geq \frac{10}{3} d_2
$$
and 
\begin{equation} \label{equ7.2}
\gamma(E) \geq \frac{1}{5} \left( \frac{10}{3} d_2 - 8 \right) = \frac{d_2}{2} + \frac{d_2}{6} - \frac{8}{5} \geq \frac{d_2}{2} -\frac{3}{5}. 
\end{equation}

Now suppose $h^0(M) = 2$.
Then $\deg M \geq d_1$ and $\deg((E/N)/M) \geq \frac{2d}{5}$ 
by semistability of $E$. So 
$$
\deg E \geq d_2 + d_1 + \frac{2d}{5} \geq 2d_2 + d_1 \geq d_2 + d_3.
$$
If $s = 2$, this gives 
\begin{equation}
\gamma(E) \geq \frac{1}{5}(d_2 + d_3 - 4).
\end{equation}
If $s=3$, then $h^0((E/N)/M) \geq 3$. So $\deg ((E/N)/M) \geq \min\{ d_2, d_1 + \frac{d}{5} \}$
by Lemma \ref{lem7.2} and
\begin{eqnarray*}
\deg E &\geq & \min \left\{ d_1 + 2d_2, 2d_1 + d_2 + \frac{d}{5} \right\}\\
& \geq & \min \left\{ d_1 + 2d_2, 2d_1 + \frac{3d_2}{2}  \right\} \geq \min \{ d_1 + 2d_2, 2d_1 + d_3 \}.
\end{eqnarray*}  
Hence
\begin{equation} \label{eq8.4}
\gamma(E) \geq \min \left\{ \frac{1}{5}(d_1 + 2d_2 - 6), \frac{1}{5}( 2d_1 + d_3 - 6) \right\}.
\end{equation}

If $s=4$, then $\deg M \geq d_1$ and $h^0((E/N)/M) \geq 4$. 
Using Lemma \ref{lem7.2}, this gives $\deg ((E/N)/M) \geq \min \{d_4, 2d_1, d_2 + \frac{d}{5} \}$. So
\begin{eqnarray*}
\deg E &\geq & \min \left\{d_2 + d_1 + d_4,  d_2 + 3d_1, 2d_2 + d_1 + \frac{d}{5} \right\}\\
& \geq &  \min \{d_5 + 6,  3d_1 + d_2, d_1 + 2d_2 + 3 \}
\end{eqnarray*}
and hence
\begin{equation}
\gamma(E) \geq \min \left\{\frac{1}{5}(d_5 - 2),  \frac{1}{5}(3d_1 + d_2 - 8), \frac{1}{5}(d_1 +2d_2 -5) \right\}.
\end{equation}

{\it Case} 1 c: Suppose $E/N$ has a subbundle $M$ of rank 2 with $h^0(M) \geq 3$ and no line subbundle with $h^0 \geq 2$.
Then $\deg M \geq d_2$ and $\deg ((E/N)/M) \geq \frac{d}{5}$. 
If $h^0(M) \geq 4$, then $\deg M \geq d_4$ by Lemma \ref{lempr}, so
$$
\deg E \geq d_2 + d_4 + \frac{d}{5} \geq d_2 + d_3 + 4
$$
and 
\begin{equation} \label{eqn7.6}
\gamma(E) \geq \frac{1}{5} (d_2 + d_3 - 4).     
\end{equation}

Now suppose $h^0(M) = 3$. Then 
$$
\deg E \geq 2d_2 + \frac{d}{5} \geq \frac{5}{2}d_2.
$$
If $s=2$, this gives 
\begin{equation} \label{eqn7.7}
\gamma(E) \geq \frac{1}{5} \left(\frac{5}{2} d_2 - 4 \right) = \frac{d_2}{2} - \frac{4}{5}.
\end{equation}
If $s=3$, then $h^0((E/N)/M) \geq 2$. So $\deg E \geq 2d_2 + d_1$ and 
\begin{equation}
\gamma(E) \geq \frac{1}{5}(d_1 + 2d_2 - 6).
\end{equation} 
If $s = 4$, then $h^0((E/N)/M) \geq 3$ and $\deg E  \geq 3d_2$. So,
\begin{equation}
\gamma(E) \geq \frac{1}{5}(3d_2 - 8) \geq \frac{1}{5}\left(\frac{5}{2}d_2 + 3 - 8\right) = \frac{1}{2}(d_2-2).
\end{equation}

{\it Case} 2: $p=3$. 
We have $h^0(N) = 4$. So Lemma \ref{lempr} implies $\deg N \geq d_3$ and by semistability, 
$\frac{d}{5} \geq \frac{d_3}{3} \geq 3$.
Then $E/N$ has rank 2 and $h^0(E/N) \geq s+1$. So we can apply Lemma \ref{lem7.2} with $t = s-1$.

If $s=2$, Lemma \ref{lem7.2} gives 
\begin{eqnarray*}
\deg E & \geq & \min \left\{ d_3 + d_2, d_3 + d_1 + \frac{d}{5} \right\} \\
& \geq & \min \left\{ d_2 + d_3, d_1 + \frac{4}{3} d_3 \right\} \geq \min \{ d_2 + d_3, d_1 + d_4\}
\end{eqnarray*}
and hence
\begin{equation} \label{eq7.14}
\gamma(E) \geq \min \left\{ \frac{1}{5}(d_2 + d_3 -4), \frac{1}{5}(d_1 + d_4 -4) \right\}.
\end{equation}

If $s=3$, Lemma \ref{lem7.2} gives 
$$
\deg E \geq \min \left\{ d_3 + d_4, d_3 + d_2 + \frac{d}{5}, d_3 + 2d_1 \right\} \geq \min \{ d_2 + d_3 + 2, 2d_1 + d_3 \}. 
$$
So 
\begin{equation}
\gamma(E) \geq \min \left\{ \frac{1}{5}(d_2 + d_3 -4), \frac{1}{5}(2d_1 + d_3 - 6) \right\}.
\end{equation}

If $s=4$, Lemma \ref{lem7.2} gives 
$$
\deg E \geq \min \left\{ d_3 + d_6, 2d_3 + \frac{d}{5}, d_3 + d_1 + d_2 \right\} \geq d_2 + d_3 + 4.
$$
So 
\begin{equation}
\gamma(E) \geq \frac{1}{5}(d_2 + d_3 -4).
\end{equation}

{\it Case} 3: $p = 4$. 
We have $h^0(N) = 5$. We can assume by Proposition \ref{prop4.18} that $h^0(E) \geq 8$.
If $h^0(E) = 8$, then $h^0(E/N) \geq 3$ and hence $\deg (E/N) \geq d_2$. So 
$$
\deg E \geq d_2 + d_4 \geq d_2 + d_3 + 1.
$$ 
Hence
\begin{equation} \label{eq8.11}
\gamma(E) \geq \frac{1}{5}(d_2 + d_3 +1 - 6) = \frac{1}{5}(d_2 + d_3 - 5).
\end{equation}

If $h^0(E) = 9$, then $h^0(E/N) \geq 4$ and hence $\deg (E/N) \geq d_3$. So
$$
\deg E \geq d_4 + d_3.
$$
If $d_4 \geq 2d_2 - 2$, then $\deg E \geq d_2 + d_3 + d_2 - 2 \geq d_2 + d_3 + 4$.
So
\begin{equation} \label{eq8.13}
\gamma(E) \geq \frac{1}{5}(d_2 + d_3 + 4 - 8) = \frac{1}{5}(d_2 + d_3 - 4).
\end{equation}
If $d_4 \leq 2d_2 -3$, then 
$$
d_5 \leq \frac{5}{4}d_4 = d_4 + \frac{1}{4}d_4 \leq d_4 + \frac{1}{2}d_2 - \frac{3}{4}.
$$
So $\deg E \geq d_5 - \frac{1}{2}d_2 + \frac{3}{4} + d_3$ and hence
\begin{equation} \label{eq8.14}
\gamma(E) = \frac{1}{5}(\deg E - 8) \geq \frac{1}{5}(d_5 - 2),
\end{equation}
provided $d_3 - \frac{1}{2}d_2 + \frac{3}{4} > 5$, i.e. $d_3 > \frac{1}{2} d_2 + \frac{17}{4}$. 
Since $d_3 \geq d_2 + 1$ and we are assuming $d_3 \geq 9$, this holds. 

The assertion follows from the inequalities \eqref{eq8.1}, \ldots, \eqref{eq8.14}.
\end{proof}

\begin{cor} \label{cor7.6}
Let $C$ be a general curve. Then
$$
\gamma_5 = \min \left\{ \gamma_5', \frac{1}{5} \left(g - \left[ \frac{g}{6} \right] + 3\right) \right\}.
$$
\end{cor}

\begin{proof}
For $g=4$ this follows from Proposition \ref{prop3.2} (a).
For $g=5$ and $g=6$ we can check it directly from Corollary \ref{cor4.15}.
For $g \geq 7$,
the inequality $\gamma_5 \geq \min \{ \gamma_5', \frac{1}{5}(g - \left[ \frac{g}{6} \right] + 3) \}$ follows by evaluating 
the numbers on the right hand side of the formula in Theorem \ref{prop8.1} using Remark \ref{rem4.2} (c). 
There exists $E$ computing $\gamma_5'$ by definition. Moreover, the conditions of 
Propositions \ref{prop5.1} (e) and \ref{cor5.2} imply the existence 
of a stable bundle $E_L$ of rank 5 and degree $d_5 = g - [\frac{g}{6}] + 5$ with $h^0(E_L) = 6$.
\end{proof}

\begin{rem}
\emph{It would be interesting to determine whether $\gamma_5' \geq \frac{1}{5}(g - [\frac{g}{6}] + 3).$} 
\end{rem}

\begin{rem} \label{rem7.6}
\emph{We do not know how many of the terms on the right hand side of the inequality in Theorem \ref{prop8.1} 
are necessary.}

\emph{Under the hypotheses of the proposition, we do know that there exists a semistable bundle $E$ of degree $d_5$ with $h^0(E) \geq 6$ 
(see Proposition \ref{prop5.1}). For this $E$,
$$
\gamma(E) \leq \frac{1}{5}(d_5 - 2).
$$}

\emph{Moreover, there exists by Proposition \ref{prop5.1} a semistable bundle $N$ of rank 2 
and degree $d_2$ with $h^0(N) \geq 3$.
Suppose $d_2$ is even and let $L$ be a line bundle of degree $\frac{d_2}{2}$ with $h^0(L) \geq 1$. Then 
$$
E = N \oplus N \oplus L
$$
is semistable of degree $\frac{5}{2} d_2$ with $h^0(E) \geq 7$. So
$$
\gamma(E) \leq \frac{1}{5} \left( \frac{5}{2}d_2 - 4 \right) = \frac{1}{2} d_2 - \frac{4}{5}.
$$
Even in this case we do not know whether there always exist bundles $E$ with $\gamma(E) \leq \frac{1}{2}d_2 - 1$.}

\emph{If $\frac{d_2}{2} = \frac{d_3}{3}$, we can take $E = N \oplus N'$, where $N$ is semistable of rank 2 and 
degree $d_2$ with $h^0(N) \geq 3$ and $N'$ is semistable of rank 3 and degree $d_3$ with $h^0(N') \geq 4$. 
Then $E$ is semistable of rank 5 with $h^0(E) \geq 7$, so
$$
\gamma(E) \leq \frac{1}{5} (d_2+d_3-4).
$$}

\emph{In an attempt to construct a semistable bundle $E$ of rank 5 with 
\begin{equation} \label{eq7.16}
\gamma(E) = \frac{1}{5}(d_2 + d_3 -5),
\end{equation} 
as in the proof of the proposition, we start with a bundle $N$ of rank 4 and degree $d_4$ with $h^0(N) = 5$ 
and a line bundle $L$ of degree $d_2$ with $h^0(L) = 3$. We consider extensions
$$
0 \ra N \ra E \ra L \ra 0.
$$
If all sections of $L$ lift to $E$, then $h^0(E) = 8$ and 
$$
\gamma(E) = \frac{1}{5}(d_2 + d_4 - 6).
$$
To achieve \eqref{eq7.16}, we need $d_4 = d_3 + 1$. 
If $d_4 \leq d_2 + 4$, then
$$
\deg E \geq 2d_4 - 4 \geq 2\frac{4}{5} d_5 -4 = d_5 + \frac{3}{5}d_5 - 4,
$$
implying that $\deg E \geq d_5 + 4$ if
$d_5 \geq 12$. Hence in this case we have
\begin{equation} \label{eq8.12}
\gamma(E) \geq \frac{1}{5}(d_5 +4 - 6) = \frac{1}{5}(d_5 - 2).
\end{equation}  
This is true for $g \geq 8$. There remains the possibility that $d_4 = d_3 + 1$ and $d_3 \geq d_2 + 4$.
If one can show that in this case we cannot have $E$ semistable with $h^0(E) = 8$, then we can 
replace $\frac{1}{5}(d_2+d_3-5)$ by $\frac{1}{5}(d_2+d_3-4)$
which looks more natural.}
\end{rem}

 \section{Plane curves}
 
 For smooth plane curves the numbers $d_r$ are known by Noether's Theorem (see \cite{ci}). For stating it, note that
 for any positive integer $r$, there are uniquely determined integers $\alpha, \beta$ with 
 $\alpha \geq 1, 0 \leq \beta \leq \alpha$ such that
 $$
 r = \frac{\alpha(\alpha + 3)}{2} - \beta.
 $$
 The reason for this is that for any $\alpha$,
 $$
 \frac{\alpha(\alpha + 3)}{2} - (\alpha + 1) = \frac{(\alpha-1)(\alpha+2)}{2}.
 $$
 {\bf Noether's Theorem}. {\it Let $C$ be a smooth plane curve of degree $\delta$. For any positive integer} $r$,
 $$
 d_r = \left\{ \begin{array}{ccl}
                \alpha \delta - \beta & \textit{if} & r < g = \frac{(\delta -1)(\delta - 2)}{2},\\
                r + \frac{(\delta -1)(\delta - 2)}{2}  & \textit{if} & r \geq g.
                \end{array} \right.
 $$               
 In particular $d_1 = \delta - 1$, $d_2 = \delta$ and $d_2$ computes $\gamma_1 = \delta -4$. Note that 
 $r < g = \frac{(\delta - 1)(\delta - 2)}{2}$ is equivalent to $\alpha \leq \delta - 3$.

\begin{prop} \label{prop9.1}
Let $C$ be a smooth plane curve of degree $\delta \geq 5$. Then\\
\emph{(a)}
$$
\gamma_2 = \left\{ \begin{array}{lcl}
                   1 & \text{if} & \delta = 5,\\
                   \frac{\delta}{2} - 1 & \text{if} & \delta \geq 6.
                   \end{array} \right.
$$
\emph{(b)} \hspace*{3.1cm} $\gamma_2' = \gamma_1 = \delta - 4.$
\end{prop}

\begin{proof} (a) is a special case of Corollary \ref{cor4.9}.
For (b), we have $\gamma_1 = \delta - 4$ and $d_4 = 2 \delta -1$.  By Theorem \ref{prop4.8},
$\gamma_2' \geq \min \{ \delta - 4, \delta - \frac{5}{2} \} = \delta - 4.$
So (b) follows from Lemma \ref{lemma2.2}.  
\end{proof}

\begin{rem}
\emph{Part (a) of Proposition \ref{prop9.1} holds also for a curve which admits as a plane model a 
general nodal curve of degree $\delta \geq 7$ with $\nu \leq \frac{1}{2}(\delta^2 -7 \delta +14)$ nodes.
This follows from the third paragraph on page 90 of \cite{c} stating that $\gamma_1 = \delta -4$ in this case.
Also $d_1 = \delta -2$ and $d_2 = \delta$.}

\emph{Hence there exist curves of all genera $g \geq 8$ with
$$
\gamma_2 = \frac{\gamma_1}{2} + 1 < \gamma_1 = d_2 - 4.
$$
In particular these curves are not general and both $d_1$ and $d_2$ compute $\gamma_1$.} 
\end{rem}

Theorem \ref{prop5.4} does not apply for plane curves, since $\frac{d_2}{2} < \frac{d_3}{3}$ for $\delta \geq 5$. Indeed,
its statement is wrong for plane curves. Instead we have
\begin{prop} \label{prop9.2}
Let $C$ be a smooth plane curve of degree $\delta \geq 5$. Then
$$
\gamma_3 = \min \left\{ \gamma_3', \frac{1}{3}\left( \left[ \frac{3\delta+1}{2} \right] - 2 \right)  \right\}.
$$
\end{prop}
\begin{proof}
If $\delta = 5$, then both sides of the equality are 1 by Proposition \ref{prop3.2} (a). So assume $\delta \geq 6$.
From the definition of $\gamma_3'$ we may assume that $E$ is a semistable bundle of rank 3 and slope $\leq g-1$ with $4 \leq h^0(E) \leq 5$.

Suppose first that $h^0(E) = 4$.
If $E$ has no proper subbundle $N$ with $h^0(N) > \rk N$, then Lemma \ref{lempr} implies that
\begin{equation*}
\deg E \geq d_3 = 2 \delta - 2.
\end{equation*}
If $E$ has a line subbundle $N$ with $h^0(N) \geq 2$, then $\deg N \geq d_1$ and so by semistability,
\begin{equation*} 
\deg E \geq 3d_1 = 3 \delta - 3.
\end{equation*}
Suppose $E$ has a subbundle $N$ of rank 2 with $h^0(N) \geq 3$ and no line subbundle with $h^0 \geq 2$.
If $h^0(N) = 4$, then $\deg N \geq d_4$ by Lemma \ref{lempr} and hence by semistability,
\begin{equation*} \label{eq9.3}
\deg E \geq \frac{3}{2}d_4 = 3 \delta - \frac{3}{2}.
\end{equation*}
If $h^0(N) = 3$, then $\deg N \geq d_2 = \delta$ by Lemma \ref{lempr} and so
\begin{equation*} \label{eq9.4}
\deg E \geq \frac{3}{2} \delta.
\end{equation*}
Hence, if $h^0(E) = 4$ we get
\begin{equation} \label{eqn9.1}
\gamma(E) = \frac{1}{3}(\deg E - 2) \geq \frac{1}{3}\left(\left[\frac{3\delta+1}{2} \right] - 2 \right).
\end{equation}

Now suppose $h^0(E) = 5$. 
If $E$ has no proper subbundle $N$ with $h^0(N) > \rk N$, then Lemma \ref{lempr} implies that
\begin{equation*} 
\deg E \geq d_6 = 3 \delta - 3.
\end{equation*}
If $E$ has a line subbundle $N$ with $h^0(N) \geq 2$, then $\deg N \geq d_1$ and so by semistability,
\begin{equation*} 
\deg E \geq 3d_1 = 3 \delta - 3.
\end{equation*} 
Suppose $E$ has a subbundle $N$ of rank 2 with $h^0(N) \geq 3$ and no line subbundle with $h^0 \geq 2$. 
If $h^0(N) \geq 4$, then $\deg N \geq d_4 = 2 \delta - 1$ by Lemma \ref{lempr} and so
\begin{equation*} 
\deg E \geq \frac{3}{2}d_4 = 3 \delta - \frac{3}{2}.
\end{equation*} 
If $h^0(N) = 3$, then $\deg N \geq d_2 = \delta$ and $\deg (E/N) \geq d_1 = \delta -1$ and so
$$
\deg E \geq 2 \delta - 1.
$$
Hence, if $h^0(E) = 5$ we get
\begin{equation} \label{eq9.2}
\gamma(E) = \frac{1}{3}(\deg E - 4) \geq \frac{1}{3}(2 \delta - 5).
\end{equation}
Since $2 \delta - 5 \geq [\frac{3\delta+1}{2}] - 2$ for $\delta \geq 6$, \eqref{eqn9.1} and \eqref{eq9.2} imply the inequality 
$\gamma_3 \geq \min \left\{ \gamma_3', \frac{1}{3}\left( \left[ \frac{3\delta+1}{2} \right] - 2 \right)  \right\}$.

To show equality we have to show that the bound 
$\frac{1}{3}\left( \left[ \frac{3\delta+1}{2} \right] - 2 \right)$ can be attained.
Since $d_1 > \frac{d_2}{2}$, Theorem \ref{pr4.9} (a) implies the existence of a semistable bundle $N$ of rank 2 and degree 
$d_2 = \delta$ with $h^0(N) = 3$.

If $\delta$ is even take an effective line bundle $M$ of degree $\frac{\delta}{2}$. Then
$$
E = N \oplus M
$$
is semistable, has degree $\frac{3\delta}{2}$ and $h^0(E) = 4$.

If $\delta$ is odd, choose an effective line bundle $M$ of degree $\frac{\delta+1}{2}$. We have $h^0(M) = 1$, 
since $d_1 = \delta -1$. Any non-zero section of $M$ induces a map
$H^1(M^* \otimes N) \ra H^1(\cO_C \otimes N)$. Comparing dimensions one checks that this map 
has a non-trivial kernel. Every non-trivial extension
$$
0 \ra N \ra E \ra M \ra 0
$$
in the kernel of this map defines a bundle $E$ of rank 3 and degree $\frac{3\delta + 1}{2}$ with $h^0(E) = 4$. 
Since any such extension is stable, this completes the proof.
\end{proof}  

\begin{prop} \label{prop9.4}
Let $C$ be a smooth plane curve of degree $\delta \geq 5$. Then 
$$
\gamma_4 = \min \left\{ \gamma_4', \frac{\delta}{2} - 1 \right\}.
$$ 
\end{prop}

\begin{proof}
We have $d_3 = 2 \delta -2$ and $d_4 = 2 \delta -1$. So $\frac{d_3}{3} > \frac{d_4}{4}$ and Theorem \ref{prop7.1} 
applies to give the assertion. 
\end{proof}

\begin{prop} \label{prop8.5}
Let $C$ be a smooth plane curve of degree $\delta \geq 5$. Then
$$
\gamma_5 = \min \left\{ \gamma_5', \frac{2}{5}(\delta -1) \right\}.
$$
\end{prop}

\begin{proof}
If $\delta = 5$, then, since $\gamma_1 = 1$, we get from Proposition \ref{prop3.2} (a) that $\gamma_5 = \gamma_5' = 1$
and the result is obvious. Suppose $\delta \geq 6$. From Noether's Theorem we get
$$
d_1 = \delta -1,\; d_2 = \delta,\; d_3 = 2\delta -2, \; d_4 = 2\delta - 1 \; \mbox{and} \; d_5 = 2 \delta.
$$
It follows that $\frac{d_p}{p} > \frac{d_5}{5}$ for all $p < 5$ and $\frac{d_p}{p} \geq \frac{d_{p+1}}{p+1}$ 
except when $p=2$. Hence Proposition \ref{prop4.11} and Lemma \ref{lem7.1} are valid. The 
only place in the proof of Theorem \ref{prop8.1}, 
where the assumption $\frac{d_2}{2} \geq \frac{d_3}{3}$ is used, is in the proof of inequality \eqref{eq8.4}. If we replace 
$\frac{1}{5}(2d_1 + d_3 -6)$ by $\frac{1}{5}(2d_1 + \frac{3}{2} d_2 -6)$, the proof is valid. With this modification, 
Theorem
\ref{prop8.1} becomes
\begin{eqnarray*}
\gamma_5 &\geq & \min \left\{  \gamma_5', \frac{1}{2}(\delta -2), \frac{1}{5}(2\delta -2), 
\frac{1}{5}(\frac{7}{2}\delta -8), \frac{1}{5}(4\delta -11) , \frac{1}{5}(3\delta-7) \right\}\\
%&& \left. \hspace*{1.2cm} \frac{1}{5}(\frac{7}{2}\delta -8), \frac{1}{5}(4\delta -11) , \frac{1}{5}(3\delta-7) \right\}\\
&=& \min \left\{ \gamma_5', \frac{1}{5}(2\delta -2) \right\}.
\end{eqnarray*}

By the definition of $\gamma_5'$ the equality $\gamma(E) = \gamma_5'$ can be attained.
The equality $\gamma(E) = \frac{1}{5} ( 2 \delta -2)$ is attained by the bundle $E_L$, where $L$ is a line bundle 
of degree $d_5 = 2 \delta$ with $h^0(L) = 6$. Since $\frac{d_p}{p} > \frac{d_5}{5}$ for all $p < 5$, the 
bundle $E_L$ is stable by Proposition \ref{prop5.1} (e) and hence $h^0(E_L) = 6$ by Theorem \ref{pr4.9}.
\end{proof}

\section{Problems}

In this section we present some problems which are related to the contents of this paper.

\begin{pr} \label{pr9.1}
Find an improved lower bound for $\gamma_5$ and good lower bounds for $\gamma_n, n \geq 6$.
\end{pr}
We expect that the term $\frac{1}{p}(d_p -2)$ for $p|n$ will appear (see Lemma \ref{lem4.10} and 
Proposition \ref{prop4.11}). There is some evidence that terms of the form 
$\frac{1}{n}(d_p + d_{n-p} - 4)$ may appear, but it is possible that a careful argument may eliminate them.
See also Remark \ref{rem7.6}.

\begin{pr}  \label{pr9.2}
Find good lower bounds for $\gamma_n'$ for $n \geq 3$.
\end{pr}

In relation to this we have the conjecture

\begin{conj} \label{con9.3}
$\gamma_n' = \gamma_1$.
\end{conj}

The conjecture is the most important point of Mercat's Conjecture \ref{conj3.1} (see Proposition \ref{prop3.3}).
If this is true, a complete proof seems a long way off. In many ways it would be more interesting if 
the conjecture were false, since this would imply the existence of new semistable bundles reflecting aspects of the geometry
of the curve $C$ not detected by classical Brill-Noether theory. A small piece of evidence in favour of the conjecture 
is presented by Proposition \ref{prop9.1} which shows that there exist curves of arbitrary $\gamma_1$ for which 
$\gamma_2' = \gamma_1$.

\begin{pr}  \label{pr9.4}
One can define Clifford indices $\gamma_n^s$ and ${\gamma_n^s}'$ by restricting to stable bundles. 
Of course $\gamma_n \leq \gamma_n^s$ and $\gamma_n' \leq {\gamma_n^s}'$. Find examples for which we have strict
inequalities or prove there are none.
\end{pr}

\begin{pr}  \label{pr9.5}
Obtain more information about the gonality sequence of a curve $C$. 
\end{pr}

This should contribute to Problem \ref{pr9.1}. For example, the knowledge of the gonality sequence for smooth plane 
curves enabled us to prove Proposition \ref{prop8.5} which is a significant improvement on Theorem \ref{prop8.1}.\\

We have seen that the classical Clifford index alone is not sufficient to determine $\gamma_n$. However, the following problem remains open.

\begin{pr}\label{pr9.5a}
Show that $\gamma_n$ is determined by the gonality sequence or find counter-examples.
\end{pr}

Let $E$ be a semistable bundle on $C$ and $V$ a subspace of $H^0(E)$ which generates $E$. Let $M_{V,E}$ be defined by \eqref{eq0.2}.
It has been conjectured by Butler \cite{bu} that for general $C$, $E$, $V$, the bundle $M_{V,E}$ is semistable. (Actually this is a slight modification of Butler's original conjecture \cite[Conjecture 2]{bu} which is set in the context of coherent systems.)
In \cite[Theorem 1.2]{b} Butler proved for any $C$ and any semistable $E$ that $M_{V,E}$ is semistable if $\mu(E) \geq 2g$. 
There are many similar results in the literature (a summary of the current state of knowledge for the case where $C$ is general and $E$ is a line bundle may be found in \cite[Section 9]{bbn}). Our Proposition \ref{prop5.1} is a further example where $C$ is not required to be general.

\begin{pr} \label{pr9.6}
Give a proof of Butler's conjecture or obtain counter-examples for either general or special curves.
\end{pr}

A solution of this conjecture would be interesting not only in its own right but because the bundles $M_{V,E}$ are related to syzygies and to Picard bundles (see in particular \cite{el,li}, where it is shown that, for sufficiently large degree, the bundle $M_{H^0(E),E}$ is, up to twisting by a line bundle, the restriction of a Picard bundle to the curve $C$ embedded in the relevant moduli space). A further observation is that, if we use \eqref{eq0.2} to map $C$ to the Grassmannian $G$  of $n$-dimensional quotients of $V$,  then $E$ and $M_{V,E}^*$ are the pullbacks of the tautological quotient bundle and subbundle respectively, so $E\otimes M_{V,E}$ is isomorphic to the pullback of the tangent bundle of $G$. Thus Butler's conjecture implies semistability of this pullback (see \cite[Theorem 4.5]{bo} for a recent result in this direction).\\

We next move on to consider extensions. Given an exact sequence
\begin{equation} \label{eq9.1}
0 \ra L \ra E \ra M \ra 0
\end{equation}
with $L$ and $M$ line bundles, there is a geometric criterion for lifting a section of $M$ to $E$ (see \cite{lna}). 
In our context this leads to several problems.

\begin{pr} \label{pr9.7}
Extend this to a usable criterion for the case when $L$ and $M$ are vector bundles.
\end{pr}

\begin{pr} \label{pr9.8}
Try to find a usable criterion for lifting several sections.
\end{pr}

Given vector bundles $L$ and $M$, the classes of nontrivial extensions \eqref{eq9.1} are parametrized by the projective space
$P = P(\mbox{Ext}^1(M,L))$. The extensions with $E$ semistable form an open set $U$ of $P$, whereas the extensions for which a
given number of independent sections of $M$ are liftable to $E$ form a closed subset $V$ of $P$.

\begin{pr} \label{pr9.9}
Determine conditions under which $U \cap V$ is non-empty.
\end{pr}
If $\dim V > \dim(P \setminus U)$, the intersection $U \cap V$ is clearly non-empty. This has been used in several papers, however 
there are many situations in which the dimensional condition does not hold.\\

\begin{pr}
Improve the lemma of Paranjape and Ramanan (Lemma \ref{lempr}) and determine conditions under which the converse is true.
\end{pr}

This would be very useful for improving some of the bounds for $\gamma_n$ and constructing 
bundles $E$ with low values for $\gamma(E)$.\\

We finish with one very specific problem.

\begin{pr} \label{pr9.10}
What is the minimal value of $d$ for which there exists a stable (semistable) bundle of rank \emph{2} with $h^0 \geq 4$?
\end{pr}

By using Lemma \ref{lempr} one can show that, in the semistable case,
$$
\min \{ 2d_1, d_4 \} \leq d \leq 2d_1
$$
(see the proof of Theorem \ref{prop4.8}), but no information beyond this seems to be available at the moment.


\begin{thebibliography}{CAV}
\bibitem{a} D. Arcara:
\emph{A lower bound for the dimension of the base locus of the generalized theta divisor}.
C. R. Math. Acad. Sci. Paris 340 (2005), no.2, 131--134.
\bibitem{ba2} E. Ballico:
\emph{Spanned vector bundles on algebraic curves and linear series}.
Rend. Istit. Mat. Univ. Trieste 27 (1995), no. 1-2, 137--156 (1996).
\bibitem{ba} E. Ballico:
\emph{Brill-Noether theory for vector bundles on projective curves}.
Math. Proc. Camb. Phil. Soc. 128 (1998), 483--499. 
\bibitem{be1} P. Belkale:
\emph{The strange duality conjecture for generic curves}.
J. Amer. Math. Soc. 21 (2008), no.1, 235--258.
\bibitem{be2} P. Belkale:
\emph{Strange duality and the Hitchin/WZW connection}.
Preprint, arXiv 0705.0717v5.
\bibitem{bbn} U. N. Bhosle, L. Brambila-Paz, P. E. Newstead:
\emph{On coherent systems of type $(n,d,n+1)$ on Petri curves}.
Manuscr. Math. 126 (2008), 409--441.
\bibitem{bp} L. Brambila-Paz:
\emph{Non-emptiness of moduli spaces of coherent systems}.
Int. J. Math. 19 (2008), 777--799.
\bibitem{bgn} L. Brambila-Paz, I. Grzegorczyk, P. E. Newstead:
\emph{Geography of Brill-Noether loci for small slopes}.
J. Alg. Geom. 6 (1997), 645--669.
\bibitem{bmno} L. Brambila-Paz, V. Mercat, P. E. Newstead, F. Ongay:
\emph{Nonemptiness of Brill-Noether loci}.
Int. J. Math. 11 (2000),737--760.
\bibitem{bo} L. Brambila-Paz, A. Ortega:
\emph{Brill-Noether bundles and coherent systems on special curves}.
Moduli spaces and vector bundles, London Math. Soc. Lecture Notes Ser., 359, Cambridge Univ. Press, Cambridge, 2009, pp.456--472.
\bibitem{b} D. C. Butler:
\emph{Normal generation of vector bundles over a curve}.
J. Diff. Geom. 39 (1994), 1--34.
\bibitem{bu} D. C. Butler:
\emph{Birational maps of moduli of Brill-Noether pairs}.
arXiv:alg-geom/9705009v1.
\bibitem{cgt} S. Casalaina-Martin, T. Gwena, M. Teixidor i Bigas:
\emph{Some examples of vector bundles in the base locus of the generalized theta divisor}.
arXiv:0707.2326v1.
\bibitem{ci} C. Ciliberto:
\emph{Alcune applicazione di un classico procedimento di Castelnuovo}.
Sem. di Geom., Dipart. di Matem., Univ. di Bologna, (1982-83), 17--43.
\bibitem{cs} J. Cilleruelo, I. Sols:
\emph{The Severi bound on sections of rank two semistable bundles on a Riemann surface}.
Ann. Math. 154 (2001), 739--758.
\bibitem{c} M. Coppens:
\emph{The gonality of general smooth curves with a prescribed plane nodal model}.
Math. Ann. 289 (1991), 89--93.
\bibitem{ckm} M. Coppens, C. Keem, G. Martens:
\emph{Primitive linear series on curves}.
Manuscr. Math. 77 (1992), 237--264.
\bibitem{cm} M. Coppens, G. Martens:
\emph{Linear series on 4-gonal curves}.
Math.  Nachr. 213 (2000), 35--55.
\bibitem{cm1} M. Coppens, G. Martens:
\emph{Linear series on a general $k$-gonal curve}.
Abh. Math. Sem. Univ. Hamburg 69 (1999), 347--371.
\bibitem{el} L. Ein, R. Lazarsfeld:
\emph{Stability and restrictions of Picard bundles, with an application to the normal bundles of elliptic curves}.
London Math. Soc. Lecture Notes Ser., 179, Cambridge Univ. Press, Cambridge, 1992, pp.149--156.
\bibitem{e} D. Eisenbud: 
\emph{Linear sections of determinantal varieties}.
Amer. J. Math. 110 (1988), 541--575.
\bibitem{elms} D. Eisenbud, H. Lange, G. Martens, F.-O. Schreyer:
\emph{The Clifford dimension of a projective curve}.
Comp. Math. 72 (1989), 173--204.
\bibitem{g} M. L. Green:
\emph{Koszul cohomology and the geometry of projective varieties}.
J. Differential Geom. 19 (1984), no.1, 125--171.
\bibitem{gl}M. L. Green, R. Lazarsfeld:
\emph{On the projective normality of complete linear series on an algebraic curve}.
Invent. Math. 83 (1985), no.1, 73--90.
\bibitem{gt}T. Gwena, M. Teixidor i Bigas:
\emph{Maps between moduli spaces of vector bundles and the base locus of the theta divisor}.
Proc. Amer. Math. Soc. 137 (2009), no.3, 853--861. 
\bibitem{k} S. Kim:
\emph{On the Clifford sequence of a general $k$-gonal curve}.
Indag. Mathem. 8 (1997), 209--216.
\bibitem{lna} H. Lange, M.S. Narasimhan:
\emph{Maximal subbundles of rank two vector bundles on curves}.
Math. Ann. 266 (1983), 55--72. 
\bibitem{ln} H. Lange, P. E. Newstead: 
\emph{On Clifford's theorem for rank-3 bundles}.
Rev. Mat. Iberoamericana 22 (2006), 287--304.
\bibitem{li} Y. Li:
\emph{Spectral curves, theta divisors and Picard bundles}.
 Int. J. Math. 2 (1991), 525--550. 
\bibitem{mo} A. Marian, D. Oprea:
\emph{The level-rank duality for non-abelian theta  functions}.
Invent. Math. 168 (2007), no.2, 225--247.
\bibitem{ms} G. Martens, F.-O. Schreyer:
\emph{Line bundles and syzygies of trigonal curves}.
Abh. Math. Sem. Univ. Hamburg 56 (1986), 169--189.
\bibitem{m1} V. Mercat: 
\emph{Le probl\`eme de Brill-Noether pour des fibr\'es stables de petite pente}.
J. reine angew. Math. 506 (1999), 1--41.
\bibitem{m2} V. Mercat:
\emph{Fibr\'es stables de pente 2}.
Bull. London Math. Soc. 33 (2001), 535--542.
\bibitem{m} V. Mercat:
\emph{Clifford's theorem and higher rank vector bundles}.
Int. J. Math. 13 (2002), 785--796.
\bibitem{pr} K. Paranjape, S. Ramanan: 
\emph{On the canonical ring of a curve}.
Algebraic Geometry and Commutative Algebra in Honor of Masayoshi Nagata (1987), 503--516. 
\bibitem{p} M. Popa:
\emph{On the base locus of the generalized theta divisor}.
C. R. Acad. Sci. Paris S\'er. I Math. 329 (1999), no.6, 507--512.
\bibitem{pro}M. Popa, M. Roth:
\emph{Stable maps and Quot schemes}.
Invent. Math. 152 (2003), no.3, 625--663.
\bibitem{re} R. Re:
\emph{Multiplication of sections and Clifford bounds for stable vector bundles on curves}.
Comm. in Alg. 26 (1998), 1931--1944. 
\bibitem{s} O. Schneider:
\emph{Sur le dimension de l'ensemble des points base du fibr\'e d\'eterminant sur $SU_C(r)$}.
Ann. Inst. Fourier (Grenoble) 57 (2007), no.2, 481--490.
\bibitem{tan} X.-J. Tan:
\emph{Clifford's theorems for vector bundles}.
Acta Math. Sin. 31 (1988), 710--720.
\bibitem{t1} M. Teixidor i Bigas:
\emph{Brill-Noether theory for stable vector bundles}.
Duke Math. J. 62 (1991), 385--400.
\bibitem{t2} M. Teixidor i Bigas:
\emph{Rank two vector bundles with canonical determinant}.
Math. Nachr. 265 (2004), 100--106.
\bibitem{t3} M. Teixidor i Bigas:
\emph{Petri map for rank two vector bundles with canonical determinant}.
Compos. Math. 144 (2008), 705--720.
\bibitem{t4} M. Teixidor i Bigas:
\emph{Syzygies using vector bundles}.
Trans. Amer. Math. Soc. 359 (2007), no.2, 897--908.
\bibitem{v} C. Voisin: 
\emph{Sur l'application de Wahl des courbes satisfaisant la condition de Brill-Noether-Petri}.
Acta Math. 168 (1992), 249--272.

\end{thebibliography}
\end{document}